\documentclass[12pt,english]{amsart}
\usepackage{amsmath,amsthm,amssymb,amsfonts,amscd, array,latexsym}
\usepackage[noadjust]{cite}
\usepackage[T2A]{fontenc}
\usepackage[cp1251]{inputenc}
\usepackage[english]{babel}
\usepackage{diagmac2} 
\usepackage{tikz}
\usetikzlibrary{positioning,calc,arrows, arrows.meta}

\usepackage{bm} 

\newtheorem{theorem}{Theorem}[section]
\newtheorem{corollary}[theorem]{Corollary}
\newtheorem{lemma}[theorem]{Lemma}
\newtheorem{proposition}[theorem]{Proposition}
\newtheorem{conjecture}[theorem]{Conjecture}

\theoremstyle{definition}
\newtheorem{definition}[theorem]{Definition}
\newtheorem{example}[theorem]{Example}
\newtheorem{problem}[theorem]{Problem}

\theoremstyle{remark}
\newtheorem{remark}[theorem]{Remark}
\numberwithin{equation}{section}

\DeclareMathOperator{\diam}{diam}
\DeclareMathOperator{\dom}{dom}

\newcommand{\RR}{\mathbb{R}}
\newcommand{\QQ}{\mathbb{Q}}

\def\<#1>{\langle #1 \rangle}

\newbox\onebox
\newcommand{\coherent}[1]{\mathbin{\setbox\onebox=\hbox{$=$}\lower0.7\ht%
\onebox\hbox{$\stackrel{#1}{=}$}}}

\begin{document}

\title[Combinatorial properties]{Combinatorial properties of ultrametrics and generalized ultrametrics}
\author{Oleksiy Dovgoshey}

\newcommand{\acr}{\newline\indent}

\address{\textbf{O. Dovgoshey}\acr
Function theory department\acr
Institute of Applied Mathematics and Mechanics of NASU\acr
Dobrovolskogo str. 1, Slovyansk 84100, Ukraine}

\email{oleksiy.dovgoshey@gmail.com}



\begin{abstract}
Let \(X\), \(Y\) be sets and let \(\Phi\), \(\Psi\) be mappings with domains \(X^{2}\) and \(Y^{2}\) respectively. We say that \(\Phi\) and \(\Psi\) are \emph{combinatorially similar} if there are bijections \(f \colon \Phi(X^2) \to \Psi(Y^{2})\) and \(g \colon Y \to X\) such that \(\Psi(x, y) = f(\Phi(g(x), g(y)))\) for all \(x\), \(y \in Y\). Conditions under which a given mapping is combinatorially similar to an ultrametric or a pseudoultrametric are found. Combinatorial characterizations are also obtained for poset-valued ultrametric distances recently defined by Priess-Crampe and Ribenboim.
\end{abstract}

\maketitle

\section{Introduction}

Recall some definitions from the theory of metric spaces. Let \(X\) be a set, let \(X^{2}\) be the Cartesian square of \(X\), 
\[
X^{2} = X \times X = \{\<x, y> \colon x, y \in X\},
\]
and let \(\RR^{+} = [0, \infty)\).

\begin{definition}\label{d1.1}
A \textit{metric} on \(X\) is a function \(d\colon X^{2} \to \RR^{+}\) such that for all \(x\), \(y\), \(z \in X\):
\begin{enumerate}
\item \(d(x,y) = 0\) if and only if \(x=y\), the \emph{positive property};
\item \(d(x,y)=d(y,x)\), the \emph{symmetric property};
\item \(d(x, y)\leq d(x, z) + d(z, y)\), the \emph{triangle inequality}.
\end{enumerate}
A metric \(d\colon X^{2} \to \RR^{+}\) is an \emph{ultrametric} on \(X\) if
\begin{enumerate}
\item [\((iv)\)] \(d(x,y) \leq \max \{d(x,z),d(z,y)\}\)
\end{enumerate}
holds for all \(x\), \(y\), \(z \in X\).
\end{definition}

Inequality \((iv)\) is often called the {\it strong triangle inequality}.

The theory of ultrametric spaces is closely connected with various investigations in mathematics, physics, linguistics, psychology and computer science. Different properties of ultrametrics have been studied in~\cite{DM2009, DD2010, DP2013SM, Groot1956, Lemin1984FAA, Lemin1984RMS39:5, Lemin1984RMS39:1, Lemin1985SMD32:3, Lemin1988, Lemin2003, QD2009, QD2014, BS2017, DM2008, DLPS2008, KS2012, Vaughan1999, Vestfrid1994, Ibragimov2012, GomoryHu(1961), Carlsson2010, DLW, Fie, GurVyal(2012), GV, Hol, H04, BH2, Lemin2003, Bestvina2002, DDP(P-adic), DP2019, DPT(Howrigid),PD(UMB), P2018(p-Adic),DP2018, DPT2015, CO2017TaAoC, Wei2017TaAoC, Ber2019SMJ}.

An useful generalization of the concept of ultrametric is the concept of pseudoultrametric and this is one of the main objects of our research below.

\begin{definition}\label{ch2:d2}
Let \(X\) be a set and let \(d \colon X^{2} \to \RR^{+}\) be a symmetric function such that \(d(x, x) = 0\) holds for every \(x \in X\). The function \(d\) is a \emph{pseudoultrametric} (\emph{pseudometric}) on \(X\) if it satisfies the strong triangle inequality (triangle inequality).
\end{definition}

The strong triangle inequality also admits a natural generalization for poset-valued mappings.

Let \((\Gamma, \leqslant)\) be a partially ordered set with the smallest element \(\gamma_0\) and let \(X\) be a nonempty set.

\begin{definition}\label{d1.3}
A mapping \(d \colon X^{2} \to \Gamma\)  is an \emph{ultrametric distance}, if the following conditions hold for all \(x\), \(y\), \(z \in X\) and \(\gamma \in \Gamma\).
\begin{enumerate}
\item [\((i)\)] \(d(x, y) = \gamma_0\) if and only if \(x = y\).
\item [\((ii)\)] \(d(x, y) = d(y, x)\).
\item [\((iii)\)] If \(d(x, y) \leqslant \gamma\) and \(d(y, z) \leqslant \gamma\), then \(d(x, z) \leqslant \gamma\).
\end{enumerate}
\end{definition}

The ultrametric distances were introduced by Priess-Crampe and Ribenboim \cite{PR1993AMSUH} and studied in~\cite{PR1996AMSUH, PR1997AMSUH, Rib2009JoA, Rib1996PMH}. This generalization of ultrametrics has some interesting applications to logic programming, computational logic and domain theory \cite{Kro2006TCS, PR2000JLP, SH1998IMSB}.

Let us recall now the definition of combinatorial similarity. In what follows we will denote by \(F(A)\) the range of a mapping \(F \colon A \to B\), \(F(A) = \{F(x) \colon x \in A\}\).

\begin{definition}[{\cite{Dov2019a}}]\label{d2.17}
Let \(X\), \(Y\) be nonempty sets and let \(\Phi\), \(\Psi\) be mappings with the domains \(X^{2}\) and \(Y^{2}\), respectively. The mapping \(\Phi\) is \emph{combinatorially similar} to \(\Psi\) if there are bijections \(f \colon \Phi(X^2) \to \Psi(Y^{2})\) and \(g \colon Y \to X\) such that
\begin{equation}\label{d2.17:e1}
\Psi(x, y) = f(\Phi(g(x), g(y)))
\end{equation}
holds for all \(x\), \(y \in Y\). In this case, we say that \(g \colon Y \to X\) is a \emph{combinatorial similarity} for the mappings \(\Psi\) and \(\Phi\).
\end{definition}

Equality~\eqref{d2.17:e1} means that the diagram
\begin{equation*}
\ctdiagram{
\ctv 0,50:{X^{2}}
\ctv 100,50:{Y^{2}}
\ctv 0,0:{\Phi(X^{2})}
\ctv 100,0:{\Psi(Y^{2})}
\ctet 100,50,0,50:{g\otimes g}
\ctet 0,0,100,0:{f}
\ctel 0,50,0,0:{\Phi}
\cter 100,50,100,0:{\Psi}
}
\end{equation*}
is commutative, where we understand the mapping \(g\otimes g\) as
\[
(g\otimes g)(\<y_1, y_2>) := \<g(y_1), g(y_2)>
\]
for \(\<y_1, y_2> \in Y^{2}\).

Some characterizations of mappings which are combinatorially similar to pseudometrics, strongly rigid pseudometrics and discrete pseudometrics were obtained in~\cite{Dov2019a}. The present paper deals with combinatorial properties of ultrametrics and generalized ultrametrics and this can be seen as a further development of research begun in~\cite{Dov2019a, Dov2019IEJA}. 

The paper is organized as follows.

In Section~2 we introduce the notions of strongly consistent mappings and \(a_0\)-coherent mappings and show that these properties of mappings are invariant w.r.t. combinatorial similarities, Proposition~\ref{p2.4}. The main results of the section, Proposition~\ref{p2.7} and Theorem~\ref{p2.10}, describe \(a_0\)-coherent mappings in terms of binary relations defined on the domains of these mappings. An important special case of combinatorial similarities, the so-called weak similarities, are introduced in Definition~\ref{d2.9} at the end of the section.

In Section~3, starting from the characterization of mappings which are combinatorially similar to pseudometrics, we prove Theorem~\ref{t3.7}, a characterization of mappings which are combinatorially similar to pseudoultrametrics with at most countable range. The corresponding results for ultrametrics are given in Corollary~\ref{c3.8}. A basic for our goals subclass of Priess-Crampe and Ribemboim ultrametric distances, the \({\preccurlyeq}_Q\)-ultra\-metrics an related them \({\preccurlyeq}_Q\)-pseudo\-ultra\-metrics, are introduced in Definition~\ref{d3.11}. In Proposition~\ref{p3.16} we show that \({\preccurlyeq}_Q\)-pseudo\-ultra\-metrics are \(a_0\)-coherent. The main result of the section is Theorem~\ref{t3.15} which gives us the necessary and sufficient condition under which a given mapping is combinatorially similar to some \({\preccurlyeq}_Q\)-pseudo\-ultra\-metric. Proposition~\ref{p3.23} and Corollary~\ref{c3.24} expand on \({\preccurlyeq}_Q\)-pseudo\-ultra\-metrics the characterization of ultrametric-preserving functions obtained recently by Pongsriiam and Termwuttipong.

Section~4 mainly describes the interrelations between combinatorial and weak similarities of \({\preccurlyeq}_Q\)-pseudo\-ultra\-metrics. First of all, in Definition~\ref{d3.13}, we expand the notion of weak similarity from usual pseudo\-ultra\-metrics to \({\preccurlyeq}_Q\)-pseudo\-ultra\-metrics. Proposition~\ref{p3.16} claims that, for all \({\preccurlyeq}_Q\)-pseudo\-ultra\-metrics, every weak similarity is a combinatorial similarity (but not conversely in general). The orders \({\preccurlyeq}_Q\), for which the weak similarities and the combinatorial similarities are the same (for the corresponding \({\preccurlyeq}_Q\)-pseudo\-ultra\-metrics) are described in Theorem~\ref{t4.3}. In Proposition~\ref{c4.4}, for every totally ordered set \((Q, {\preccurlyeq}_Q)\) (which contains a smallest element) we construct a \({\preccurlyeq}_Q\)-ultra\-metric satisfying conditions of Theorem~\ref{t4.3}. Using this result in Proposition~\ref{p4.8} we found a metric \(d^{*}\), defined on a set \(X\) with \(|X| = 2^{\aleph_{0}}\), such that \(d^{*}\) is not combinatorially similar to any ultrametric but, for every countable \(X_1 \subseteq X\), the restriction \(d^{*}\) on \(X_1\) is combinatorially similar to an ultrametric.
The mappings which are combinatorially similar to \({\preccurlyeq}_Q\)-pseudo\-ultra\-metrics are described in Theorems~\ref{t4.11}, \ref{t4.15} and \ref{t4.19} for the case of totally ordered \((Q, {\preccurlyeq}_Q)\) satisfying the distinct universal and topological restrictions. 
The final results of the paper, Theorem~\ref{t4.20} and Corollary~\ref{c4.22}, give a kind of necessary and sufficient conditions under which a given mapping is combinatorially similar to a pseudoultrametric or, respectively, to an ultrametric.

\section{Consistency with equivalence relations}

Let \(X\) be a set. A \emph{binary relation} on \(X\) is a subset of the Cartesian square \(X^{2}\). A relation \(R \subseteq X^{2}\) is an \emph{equivalence relation} on \(X\) if the following conditions hold for all \(x\), \(y\), \(z \in X\):
\begin{enumerate}
\item \(\<x, x> \in R\), the \emph{reflexive} property;
\item \((\<x, y> \in R) \Leftrightarrow (\<y, x> \in R)\), the \emph{symmetric} property;
\item \(((\<x, y> \in R) \text{ and } (\<y, z> \in R)) \Rightarrow (\<x, z> \in R)\), the \emph{transitive} property.
\end{enumerate}

Let \(R\) be an equivalence relation on \(X\). A mapping \(F \colon X^{2} \to X\) is \emph{consistent} with \(R\) if the implication
\begin{equation*}
\bigl(\<x_1, x_2> \in R \text{ and } \<x_3, x_4> \in R\bigr) \Rightarrow \bigl(\<F(x_1, x_3), F(x_2, x_4)> \in R\bigr)
\end{equation*}
is valid for all \(x_1\), \(x_2\), \(x_3\), \(x_4 \in X\) (see~\cite[p.~78]{KurMost}). Similarly, we will say that a mapping \(\Phi \colon X^{2} \to Y\) is \emph{strongly consistent} with \(R\) if  the implication
\begin{equation}\label{e2.5}
\bigl(\<x_1, x_2> \in R \text{ and } \<x_3, x_4> \in R\bigr) \Rightarrow \bigl(\Phi(x_1, x_3) = \Phi(x_2, x_4)\bigr)
\end{equation}
is valid for all \(x_1\), \(x_2\), \(x_3\), \(x_4 \in X\).

\begin{remark}\label{r2.1}
Let \(R\) be an equivalence relation on a set \(X\). Then every strongly consistent with \(R\) mapping \(\Phi \colon X^{2} \to X\) is consistent with \(R\). The converse statement holds if and only if \(R\) is the diagonal of \(X\),
\[
R = \Delta_{X} = \{\<x, x> \colon x \in X\}.
\]
\end{remark}

\begin{definition}\label{d2.5}
Let \(X\) be a nonempty set, let \(\Phi\) be a mapping with \(\dom \Phi = X^{2}\) and let \(a_0 \in \Phi(X^{2})\). The mapping \(\Phi\) is \(a_0\)-\emph{coherent} if \(\Phi\) is strongly consistent with the fiber 
\[
\Phi^{-1}(a_0) := \{\<x, y> \colon \Phi(x, y) = a_0\}.
\]
\end{definition}

\begin{remark}\label{r2.3}
In particular, if \(\Phi\) is \(a_0\)-coherent, then \(\Phi^{-1}(a_0)\) is an equivalence relation on \(X\).
\end{remark}

The following proposition claims that the properties to be strongly consistent and to be coherent are invariant w.r.t. combinatorial similarities.

\begin{proposition}\label{p2.4}
Let \(X\), \(Y\) be nonempty sets, let \(\Phi\), \(\Psi\) be combinatorially similar mappings with \(\dom \Phi = X^{2}\) and \(\dom \Psi = Y^{2}\) and the commutative diagram
\begin{equation*}
\ctdiagram{
\def\y{25}
\ctv 0,\y:{X^{2}}
\ctv 100,\y:{Y^{2}}
\ctv 0,-\y:{\Phi(X^{2})}
\ctv 100,-\y:{\Psi(Y^{2})}
\ctet 100,\y,0,\y:{g\otimes g}
\ctet 0,-\y,100,-\y:{f}
\ctel 0,\y,0,-\y:{\Phi}
\cter 100,\y,100,-\y:{\Psi}
}.
\end{equation*}
If \(\Phi\) is strongly consistent with an equivalence relation \(R_X\) on \(X\), then \(\Psi\) is strongly consistent with an equivalence relation \(R_Y\) on \(Y\) satisfying
\[
(\<x,y> \in R_Y) \Leftrightarrow (\<g(x), g(y)> \in R_X)
\]
for every \(\<x, y> \in Y^{2}\). In addition, if \(\Phi\) is \(a_0\)-coherent for \(a_0 \in \Phi(X^{2})\), then \(\Psi\) is \(f(a_0)\)-coherent.
\end{proposition}

The proof is straightforward and we omit it here.
\medskip

Let \(X\) be a set and let \(R_1\) and \(R_2\) be binary relations on \(X\). Recall that a composition of binary relations \(R_1\) and \(R_2\) is a binary relation \(R_1 \circ R_2 \subseteq X^{2}\) for which \(\<x, y> \in R_1 \circ R_2\) holds if and only if there is \(z \in X\) such that \(\<x, z> \in R_1\) and \(\<z, y> \in R_2\). 

Using the notion of binary relations composition we can reformulate Definition~\ref{d2.5} as follows.

\begin{proposition}\label{p2.7}
Let \(X\) be a nonempty set, \(\Phi\) be a mappings with \(\dom \Phi = X^{2}\) and let \(a_0 \in \Phi(X^{2})\). Then \(\Phi\) is \(a_0\)-coherent if and only if the fiber \(R = \Phi^{-1}(a_0)\) is an equivalence relation on \(X\) and the equality
\begin{equation}\label{p2.7:e1}
\Phi^{-1}(b) = R \circ \Phi^{-1}(b) \circ R
\end{equation}
holds for every \(b \in \Phi(X^{2})\).
\end{proposition}

\begin{proof}
It suffices to show that \(\Phi\) is strongly consistent with \(R\) if and only if equality \eqref{p2.7:e1} holds for every \(b \in \Phi(X^{2})\). Let \(b \in \Phi(X^{2})\) and \eqref{p2.7:e1} hold. Suppose \(\<x_1, x_3> \in X^{2}\) such that
\[
\Phi(x_1, x_3) = b.
\]
If \(\<x_2, x_1> \in R\),  \(\<x_1, x_3> \in \Phi^{-1}(b)\) and \(\<x_3, x_4> \in R\), then from the definition of the composition \(\circ\) we obtain
\[
\<x_2, x_4> \in R \circ \Phi^{-1}(b) \circ R
\]
that implies \(\<x_2, x_4> \in \Phi^{-1}(b)\) by equality~\eqref{p2.7:e1}. Thus, the implication~\eqref{e2.5} is valid.

Conversely, suppose that \(\Phi\) is strongly consistent with \(R\). Then~\eqref{e2.5} implies the inclusion
\begin{equation}\label{p2.7:e3}
R \circ \Phi^{-1}(b) \circ R \subseteq \Phi^{-1}(b)
\end{equation} 
for every \(b \in \Phi(X^{2})\). Since \(R\) is reflexive, the converse inclusion is also valid. Equality~\eqref{p2.7:e1} follows.
\end{proof}

\begin{corollary}\label{c2.5}
Let \(X\) be a nonempty set, let \(\Phi\) be a symmetric mapping with \(\dom \Phi = X^{2}\) and let \(a_0 \in \Phi(X^{2})\). Suppose \(R := \Phi^{-1}(a_0)\) is an equivalence relation on \(X\). Then the following conditions are equivalent.
\begin{enumerate}
\item[\((i)\)] \(\Phi\) is \(a_0\)-coherent.
\item[\((ii)\)] \(\Phi^{-1}(b) = R \circ \Phi^{-1}(b) \circ R\) holds for every \(b \in \Phi(X^{2})\).
\item[\((iii)\)] \(\Phi^{-1}(b) = R \circ \Phi^{-1}(b)\) holds for every \(b \in \Phi(X^{2})\).
\item[\((iv)\)] \(\Phi^{-1}(b) = \Phi^{-1}(b) \circ R\) holds for every \(b \in \Phi(X^{2})\).
\item[\((v)\)] For every \(b \in \Phi(X^{2})\), at least one of the equalities 
\[
\Phi^{-1}(b) = R \circ \Phi^{-1}(b), \quad \Phi^{-1}(b) = \Phi^{-1}(b) \circ R
\]
holds.
\end{enumerate}
\end{corollary}

\begin{proof}
In what follows, for every \(b \in \Phi(X^{2})\), we write \(R_b = \Phi^{-1}(b)\) and, for every \(A \subseteq X^{2}\), define the inverse binary relation \(A^{T}\) by the rule:
\begin{itemize}
\item the membership \(\<x, y> \in A^{T}\) holds if and only if \(\<y, x> \in A\).
\end{itemize}

Suppose \((v)\) is valid and we have 
\begin{equation}\label{c2.5:e1}
R_b = R_b \circ R.
\end{equation}
It is trivial that a binary relation \(A\) is symmetric if and only if we have \(A^{T} = A\). Furthermore, the equality
\[
(C \circ B)^{T} = B^{T} \circ C^{T}
\]
holds for all binary relations \(B\) and \(C\) defined on the one and the same set (see, for example, \cite[p.~15]{How1976AP}). Consequently, from \eqref{c2.5:e1} it follows that
\begin{align*}
R_b &= (R_b)^{T} = (R_b \circ R)^{T} = R^{T} \circ R_b^{T} \\
& = R \circ R_b = R \circ (R_b \circ R) = R \circ R_b \circ R.
\end{align*}
Similarly, from \(R_b = R \circ R_b\) follows \(R_b = R \circ R_b \circ R\). Thus, the implication \((v) \Rightarrow (ii)\) is valid.

If \((ii)\) holds, then we have
\[
R_b = R \circ R_b \circ R
\]
for every \(b \in \Phi(X^{2})\). Since \(R\) is an equivalence relation, the equality \(R \circ R = R\) holds. Consequently,
\begin{align*}
R_b & = (R \circ R) \circ R_b \circ R = R \circ (R \circ R_b \circ R) = R \circ R_b.
\end{align*}
Thus, \((ii)\) implies \((iii)\). Analogously, \((ii)\) implies \((iv)\). The implications \((iii) \Rightarrow (v)\) and \((iv) \Rightarrow (v)\) are evidently valid. To complete the proof we recall that \((i)\) and \((ii)\) are equivalent by Proposition~\ref{p2.7}.
\end{proof}

Let \(X\) be a nonempty set and \(P = \{X_j \colon j \in J\}\) be a set of nonempty subsets of \(X\). Then \(P\) is a \emph{partition} of \(X\) with the blocks \(X_j\) if
\[
\bigcup_{j \in J} X_j = X
\]
and \(X_{j_1} \cap X_{j_2} = \varnothing\) holds for all distinct \(j_1\), \(j_2 \in J\). 

There exists the well-known, one-to-one correspondence between the equivalence relations and partitions.

If \(R\) is an equivalence relation on \(X\), then an \emph{equivalence class} is a subset \([a]_R\) of \(X\) having the form
\begin{equation}\label{e1.1}
[a]_R = \{x \in X \colon \<x, a> \in R\}, a \in X.
\end{equation}
The \emph{quotient set} of \(X\) w.r.t. \(R\) is the set of all equivalence classes \([a]_R\), \(a \in X\).

\begin{proposition}\label{p2.5}
Let \(X\) be a nonempty set. If \(P = \{X_j \colon j \in J\}\) is a partition of \(X\) and \(R_P\) is a binary relation on \(X\) defined as
\begin{itemize}
\item[] \(\<x, y> \in R_P\) if and only if \(\exists j \in J\) such that  \(x \in X_j\) and \(y \in X_j\),
\end{itemize}
then \(R_P\) is an equivalence relation on \(X\) with the equivalence classes \(X_j\). Conversely, if \(R\) is an equivalence relation on \(X\), then the set \(P_R\) of all distinct equivalence classes \([a]_R\) is a partition of \(X\) with the blocks \([a]_R\).
\end{proposition}

For the proof, see, for example, \cite[Chapter~II, \S{}~5]{KurMost}.

\begin{lemma}[{\cite[p.~9]{Kel1975S}}]\label{l2.6}
Let \(X\) be a nonempty set. If \(R\) is an equivalence relation on \(X\) and \(P_R = \{X_j \colon j \in J\}\) is the corresponding partition of \(X\), then the equality
\begin{equation*}
R = \bigcup_{j \in J} X_j^2
\end{equation*}
holds.
\end{lemma}

For every partition \(P = \{X_j \colon j \in J\}\) of a nonempty set \(X\) we define a partition \(P \otimes P^1\) of \(X^{2}\) by the rule: 
\begin{itemize}
\item A subset \(B\) of \(X^{2}\) is a block of \(P \otimes P^1\) if and only if either 
\[
B = \bigcup_{j \in J} X_{j}^{2}
\]
or there are \emph{distinct} \(j_1\), \(j_2 \in J\) such that
\[
B = X_{j_1} \times X_{j_2}.
\]
\end{itemize}

\begin{definition}\label{d2.8}
Let \(X\) be a nonempty set and let \(P_1\) and \(P_2\) be partitions of \(X\). The partition \(P_{1}\) is \emph{finer} than the partition \(P_{2}\) if the inclusion 
\[
[x]_{R_{P_1}} \subseteq [x]_{R_{P_2}}
\]
holds for every \(x \in X\), where \(R_{P_1}\) and \(R_{P_2}\) are equivalence relations corresponding to \(P_1\) and \(P_2\) respectively.
\end{definition}

If \(P_1\) is finer than \(P_2\), then we say that \(P_{1}\) is a \emph{refinement} of \(P_{2}\).

The following proposition gives us a new characterization of \(a_0\)-coherent mappings.

\begin{theorem}\label{p2.10}
Let \(X\) be a nonempty set, \(\Phi\) be a mapping with \(\dom \Phi= X^{2}\) and let \(a_0 \in \Phi(X^{2})\). Then \(\Phi\) is \(a_0\)-coherent if and only if the fiber 
\[
R := \Phi^{-1}(a_0)
\]
is an equivalence relation on \(X\) and the partition \(P_R \otimes P_R^1\) of \(X^{2}\) is a refinement of the partition \(P_{\Phi^{-1}} := \{\Phi^{-1}(b) \colon b \in \Phi(X^{2})\}\), where \(P_R\) is a partition of \(X\) whose blocks are the equivalence classes of \(R\).
\end{theorem}

\begin{proof}
Let \(\Phi\) be \(a_0\)-coherent. Then, by Definition~\ref{d2.5}, \(R\) is an equivalence relation on \(X\). We claim that \(P_R \otimes  P_R^{1}\) is a refinement \(P_{\Phi^{-1}}\). It suffices to show that for every block \(B_0\) of \(P_R \otimes  P_R^{1}\) there is \(b_0 \in \Phi(X^{2})\) such that
\begin{equation}\label{p2.10:e1}
B_0 \subseteq \Phi^{-1}(b_0).
\end{equation}
Suppose that 
\begin{equation}\label{p2.10:e2}
B_0 = \bigcup_{j \in J} X_j^{2},
\end{equation}
where \(X_j\), \(j \in J\), are the blocks of the partition corresponding to the equivalence relation \(\Phi^{-1}(a_0)\) on \(X\). By Lemma~\ref{l2.6}, we have the equality
\[
\bigcup_{j \in J} X_j^{2} = \Phi^{-1}(a_0).
\]
The last equality and \eqref{p2.10:e2} imply \eqref{p2.10:e1} with \(b_0 = a_0\). If \(B_0\) is a block of \(P_R \otimes  P_R^{1}\) but \eqref{p2.10:e2} does not hold, then there are two distinct \(j_1\), \(j_2 \in J\) such that
\begin{equation}\label{p2.10:e3}
B_0 = X_{j_1} \times X_{j_2}.
\end{equation}
Let \(x_1 \in X_{j_1}\) and \(x_2 \in X_{j_2}\) and let \(b_0 \in \Phi(X^{2})\) such that
\begin{equation}\label{p2.10:e4}
\<x_1, x_2> \in \Phi^{-1}(b_0).
\end{equation}
We must show that 
\begin{equation}\label{p2.10:e5}
X_{j_1} \times X_{j_2} \subseteq \Phi^{-1}(b_0).
\end{equation}
It follows from Proposition~\ref{p2.7} and Lemma~\ref{l2.6} that
\begin{equation}\label{p2.10:e6}
\Phi^{-1}(b_0) = \left(\bigcup_{j \in J} X_j^{2}\right) \circ \Phi^{-1}(b_0) \circ \left(\bigcup_{j \in J} X_j^{2}\right)
\end{equation}
holds. Inclusion~\eqref{p2.10:e5} holds if, for every \(x \in X_{j_1}\) and \(y \in X_{j_2}\), we have
\[
\<x, y> \in \Phi^{-1}(b_0).
\]
Using \eqref{p2.10:e6} we obtain
\begin{equation}\label{p2.10:e7}
\Phi^{-1}(b_0) \supseteq X_{j_1}^{2} \circ \Phi^{-1}(b_0) \circ X_{j_2}^{2}.
\end{equation}
Since \(\<x, x_1> \in X_{j_1}^{2}\) and \(\<x_1, x_2> \in \Phi^{-1}(b_0)\) and \(\<x_2, y> \in X_{j_2}^{2}\), the definition of composition \(\circ\) and \eqref{p2.10:e7} imply \(\<x, y> \in \Phi^{-1}(b_0)\). Thus, \(P_R \otimes  P_R^{1}\) is a refinement of \(P_{\Phi^{-1}}\) if \(\Phi\) is \(a_0\)-coherent.

Conversely, suppose that \(R = \Phi^{-1}(a_0)\) is an equivalence relation on \(X\) and \(P_R \otimes  P_R^{1}\) is a finer than \(P_{\Phi^{-1}}\). By Proposition~\ref{p2.7}, the mapping \(\Phi\) is \(a_0\)-coherent if and only if the equality
\[
R \circ \Phi^{-1}(b) \circ R = \Phi^{-1}(b)
\]
holds for every \(b \in \Phi(X^{2})\). The reflexivity of \(R\) implies that 
\[
R \circ \Phi^{-1}(b) \circ R \supseteq \Phi^{-1}(b).
\]
Consequently, to complete the proof it suffices to show that
\begin{equation}\label{p2.10:e8}
R \circ \Phi^{-1}(b) \circ R \subseteq \Phi^{-1}(b)
\end{equation}
holds for every \(b \in \Phi(X^{2})\). Inclusion~\eqref{p2.10:e8} holds if and only if 
\begin{equation}\label{p2.10:e9}
R \circ \{\<x, y>\} \circ R \subseteq \Phi^{-1}(b)
\end{equation}
holds for every \(\<x, y> \in \Phi^{-1}(b)\), where \(\{\<x, y>\}\) is the one-point subset of \(X^{2}\) consisting the point \(\<x, y>\) only. A simple calculation shows that
\begin{equation}\label{p2.10:e10}
B = R \circ B \circ R
\end{equation}
holds for every block \(B\) of the partition \(P_R \otimes  P_R^{1}\). Since \(P_R \otimes  P_R^{1}\) is a refinement of \(P_{\Phi^{-1}}\), equality \eqref{p2.10:e10} implies \eqref{p2.10:e9} for \(\<x, y> \in B\).
\end{proof}

Let us consider now some examples.

\begin{proposition}\label{c2.2}
Let \(X\) be a nonempty set and let \(d \colon X^{2} \to \RR^{+}\) be a pseudoultrametric on \(X\). Then \(d^{-1}(0)\) is an equivalence relation on \(X\) and \(d\) is \(0\)-coherent.
\end{proposition}

This proposition is a corollary of the corresponding result for pseudometrics \cite[Ch.~4, Th.~15]{Kel1975S}.

\begin{definition}\label{d2.9}
Let \((X_1, d_1)\) and \((X_2, d_2)\) be pseudoultrametric spaces. A bijection \(\Phi \colon X_1 \to X_2\) is a weak similarity if there is a strictly increasing bijective function \(f \colon d_1(X_{1}^{2}) \to d_2(X_{2}^{2})\) such that the equality
\begin{equation}\label{d2.9:e1}
d_1(x, y) = f(d_2(\Phi(x), \Phi(y)))
\end{equation}
holds for all \(x\), \(y \in X_1\).
\end{definition}

\begin{remark}\label{r2.10}
The weak similarities of semimetric spaces and ultrametric ones were studied in~\cite{DP2013AMH} and \cite{P2018(p-Adic)}. See also~\cite{KvL2014} and references therein for some results related to weak similarities of subsets of Euclidean finite-dimensional spaces.
\end{remark}

\begin{proposition}\label{p2.11}
Let \((X_1, d_1)\) and \((X_2, d_2)\) be pseudoultrametric spaces and \(\Phi \colon X_1 \to X_2\) be a weak similarity. Then \(\Phi\) is a combinatorial similarity for the pseudoultrametrics \(d_1\) and \(d_2\).
\end{proposition}

\begin{proof}
It follows directly from Definition~\ref{d2.9} and Definition~\ref{d2.17}.
\end{proof}

\section{Combinatorial similarity for generalized ultrametrics}

First of all, we recall a combinatorial characterization of arbitrary pseudometric.

\begin{theorem}[{\cite{Dov2019a}}]\label{ch2:p7}
Let \(X\) be a nonempty set. The following conditions are equivalent for every mapping \(\Phi\) with \(\dom\Phi = X^{2}\).
\begin{enumerate}
\item\label{ch2:p7:s1} \(\Phi\) is combinatorially similar to a pseudometric.
\item\label{ch2:p7:s2} \(\Phi\) is symmetric, and \(|\Phi(X^{2})| \leqslant 2^{\aleph_{0}}\), and there is \(a_0 \in \Phi(X^{2})\) such that \(\Phi\) is \(a_0\)-coherent.
\end{enumerate}
\end{theorem}

\begin{corollary}[{\cite{Dov2019a}}]\label{c3.2}
Let \(X\) be a nonempty set and let \(\Phi\) be a mapping with \(\dom\Phi = X^{2}\). Then \(\Phi\) is combinatorially similar to a metric if and only if \(\Phi\) is symmetric, and \(|\Phi(X^{2})| \leqslant 2^{\aleph_{0}}\), and there is \(a_0 \in \Phi(X^{2})\) such that \(\Phi^{-1}(a_0) = \Delta_{X}\), where \(\Delta_{X}\) is the diagonal of \(X\).
\end{corollary}

Consequently, if a mapping \(\Phi\), with \(\dom \Phi = X^{2}\), is combinatorially similar to a pseudoultrametric, then it satisfies condition \((ii)\) of Theorem~\ref{ch2:p7}.

Another necessary condition for combinatorial similarity of \(\Phi\) to a pseudoultrametric follows from the fact that
\begin{itemize}
\item all triangles are isosceles in every pseudoultrametric space.
\end{itemize}

This fact can be written in the form.

\begin{lemma}\label{l3.1}
Let \(X\) be a nonempty set and let \(\Phi\) be a mapping with \(\dom \Phi = X^{2}\). If \(\Phi\) is combinatorially similar to a pseudoultrametric, then,
\begin{enumerate}
\item[\((i)\)] for every triple \(\<x_1, x_2, x_3>\) of points of \(X\), there is a permutation
\[
\begin{pmatrix}
x_1 & x_2 & x_3\\
x_{i_1} & x_{i_2} & x_{i_3}
\end{pmatrix}
\]
such that \(\Phi(x_{i_1}, x_{i_2}) = \Phi(x_{i_2}, x_{i_3})\).
\end{enumerate}
\end{lemma}

The following example shows that condition~\((i)\) is not sufficient for existence of a pseudoultrametric \(d\) which is combinatorially similar to \(\Phi\), even \(\Phi\) is a metric.

\begin{example}\label{ex3.4}
Let \(X = \{x_1, x_2, x_3, x_4\}\) and let \(\rho \colon X^{2} \to \RR^{+}\) be a symmetric mapping defined as
\begin{equation}\label{ex3.4:e1}
\rho(x, y) = \begin{cases}
0 & \text{if } x = y,\\
\frac{\pi}{2} & \text{if } \{x, y\} = \{x_1, x_2\} \text{ or } \{x, y\} = \{x_2, x_3\},\\
\pi & \text{otherwise}.
\end{cases}
\end{equation}
It is easy to see that \(\rho\) is a metric on \(X\) such that every triangle is isosceles in \((X, \rho)\) (see Figure~\ref{fig1}). Suppose \(\rho\) is combinatorially similar to some pseudoultrametric \(d \colon Y^{2} \to \RR^{+}\). Then, by Definition~\ref{d2.17}, there are bijections \(f \colon \rho(X^{2}) \to d(Y^{2})\) and \(g \colon Y \to X\) such that
\[
d(x, y) = f(\rho(g(x), g(y)))
\]
for all \(x\), \(y \in Y\). The last equality and \eqref{ex3.4:e1} imply
\[
d(g^{-1}(x_1), g^{-1}(x_2)) = d(g^{-1}(x_2), g^{-1}(x_3)) = f\left(\frac{\pi}{2}\right)
\]
and
\begin{equation*}
d(g^{-1}(x_1), g^{-1}(x_4)) = d(g^{-1}(x_4), g^{-1}(x_2)) = f\left(\pi\right).
\end{equation*}
Using these equalities and the strong triangle inequality (for the triples \(\<g^{-1}(x_1), g^{-1}(x_2), g^{-1}(x_3)>\) and \(\<g^{-1}(x_1), g^{-1}(x_4), g^{-1}(x_2)>\)) we obtain
\[
f\left(\frac{\pi}{2}\right) \geqslant f\left(\pi\right) \text{ and } f\left(\frac{\pi}{2}\right) \leqslant f\left(\pi\right).
\]
Thus, \(f\left(\frac{\pi}{2}\right) = f\left(\pi\right)\) holds, contrary to the bijectivity of \(f\).
\end{example}

\begin{figure}[htb]
\begin{tikzpicture}[scale=1]
\draw (1,0) arc (0:360:1cm and 0.5cm);
\draw[arrows={-Stealth[length=10pt]}] (0,2) -- (-1,0);
\draw[arrows={-Stealth[length=10pt]}] (0,2) -- (0,-0.5);
\draw[arrows={-Stealth[length=10pt]}] (0,2) -- (1,0);
\draw [fill=black] (-1,0) circle (2pt) node [label=left:\(x_1\)] {};
\draw [fill=black] (0,-0.5) circle (2pt) node [label=below:\(x_2\)] {};
\draw [fill=black] (1,0) circle (2pt) node [label=right:\(x_3\)] {};
\draw [fill=black] (0,2) circle (2pt) node [label=above:\(x_4\)] {};
\end{tikzpicture}
\caption{The metric space \((X, \rho)\) is (up to isometry) a subspace of the metric space \(L\) consisting of the three rays \(\protect\overrightarrow{x_4x_1}\), \(\protect\overrightarrow{x_4x_2}\), \(\protect\overrightarrow{x_4x_3}\) and a unit circle (a circle with the radius \(1\)) passing through \(x_1\), \(x_2\) and \(x_3\) if we consider \(L\) endowed with the shortest path metric.}
\label{fig1}
\end{figure}
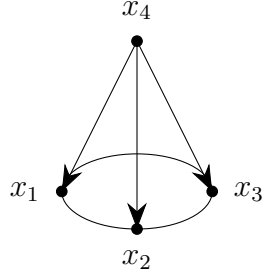

We want to describe the mappings which are combinatorially similar to pseudoultrametrics. For this goal we recall some definitions.

Let \(\gamma\) be a binary relation on a set \(X\). We will write \(\gamma^{1} = \gamma\) and \(\gamma^{n+1} =  \gamma^{n}\circ \gamma\) for every integer \(n \geqslant 1\). The \emph{transitive closure} \(\gamma^{t}\) of \(\gamma\) is the relation
\begin{equation}\label{e3.2}
\gamma^{t} := \bigcup_{n=1}^{\infty} \gamma^{n}.
\end{equation}

For every \(\beta \subseteq X^{2}\), the transitive closure \(\beta^{t}\) is transitive and the inclusion \(\beta \subseteq \beta^{t}\) holds. Moreover, if \(\tau \subseteq X^{2}\) is an arbitrary transitive binary relation for which \(\beta \subseteq \tau\), then we also have \(\beta^{t} \subseteq \tau\), i.e., \(\beta^{t}\) is the smallest transitive binary relation containing \(\beta\).

Recall that a reflexive and transitive binary relation \(\preccurlyeq_Y\) on a set \(Y\) is a \emph{partial order} on \(Y\) if, for all \(x\), \(y \in Y\), we have the \emph{antisymmetric property},
\[
\bigl(\<x, y> \in \preccurlyeq_Y \text{ and } \<y, x> \in \preccurlyeq_Y \bigr) \Rightarrow (x = y).
\]

In what follows we use the formula \(x \preccurlyeq y\) instead of \(\<x, y> \in \preccurlyeq\) and write \(x\prec y\) instead of
\[
x \preccurlyeq y \quad \text{and} \quad x \neq y.
\]

Let \(\preccurlyeq_Y\) be a partial order on a set \(Y\). A pair \((Y, \preccurlyeq_Y)\) is called to be a \emph{poset} (a partially ordered set). A poset \((Y, \preccurlyeq_Y)\) is \emph{linear} (= \emph{totally ordered}) if, for all \(y_1\), \(y_2 \in Y\), we have 
\[
y_1 \preccurlyeq_Y y_2 \quad \text{or} \quad y_2 \preccurlyeq_Y y_1.
\]

\begin{definition}\label{d3.5}
Let \((Q, {\preccurlyeq}_Q)\) and \((L, {\preccurlyeq}_L)\) be posets. A mapping \(f \colon Q \to L\) is \emph{isotone} if, for all \(q_1\), \(q_2 \in Q\), we have
\[
(q_1 \preccurlyeq_Q q_2) \Rightarrow (f(q_1) \preccurlyeq_L f(q_2)).
\]

Let \(\Phi \colon X \to Y\) be an isotone mapping of posets \((X, {\preccurlyeq}_X)\) and \((Y, {\preccurlyeq}_Y)\). If \(\Phi\) is bijective and the inverse mapping \(\Phi^{-1} \colon Y \to X\) is also isotone, then we say that \((X, {\preccurlyeq}_X)\) and \((Y, {\preccurlyeq}_Y)\) are \emph{isomorphic} and \(\Phi\) is an (\emph{order}) \emph{isomorphism}.
\end{definition}

If \((Y, \preccurlyeq_Y)\) is a poset, and \(Y_1 \subseteq Y\), and \(\preccurlyeq_{Y_1}\) is a partial order on \(Y_1\) such that, for all \(x\), \(y \in Y_1\),
\[
(x \preccurlyeq_{Y_1} y) \Leftrightarrow (x \preccurlyeq_Y y),
\]
then we say that \((Y_1, \preccurlyeq_{Y_1})\) is a \emph{subposet} of the poset \((Y, \preccurlyeq_Y)\). 

Write \(\mathbb{Q}^{+}\) for the set of all nonnegative rational numbers, 
\[
\mathbb{Q}^{+} = \mathbb{Q} \cap [0, +\infty),
\]
and let \(\leqslant\) be the usual ordering on \(\mathbb{Q}^{+}\).

\begin{lemma}[Cantor]\label{l3.2}
Let \((X, \preccurlyeq_X)\) be a totally ordered set and let \(|X| \leqslant \aleph_{0}\) hold. Then \((X, \preccurlyeq_X)\) is isomorphic to a subposet of \((\mathbb{Q}^{+}, \leqslant)\).
\end{lemma}

The proof can be obtained directly from the classical Cantor's results (see, for example, \cite{Ros1982}, Chapter~2, Theorem~2.6 and Theorem~2.8).

We will also use the following Szpilrajn Theorem.

\begin{lemma}[Szpilrajn]\label{l3.3}
Let \((X, \preccurlyeq_X)\) be a poset. Then there is a linear order \(\preccurlyeq\) on \(X\) such that
\[
{\preccurlyeq_X} \subseteq {\preccurlyeq}.
\]
\end{lemma}

Informally speaking it means that each partial order on a set can be extended to a linear order on the same set.

\begin{remark}\label{r3.9}
This result was obtained by Edward Szpilrajn in~\cite{Szp1930FM}. Interesting reviews of Szpilrajn-type theorems can be found in~\cite{And2009} and \cite{BP1982}.
\end{remark}

Let \(X\) be a nonempty set and let \(\Phi\) be a symmetric mapping with \(\dom\Phi = X^{2}\) and let \(Y := \Phi(X^{2})\). Let us define a binary relation \(u_{\Phi}\) by the rule: \(\<y_1, y_2> \in u_{\Phi}\) if and only if \(\<y_1, y_2> \in Y^{2}\) and there are \(x_1\), \(x_2\), \(x_3 \in X\) such that
\begin{equation}\label{e2.16}
y_1 = \Phi(x_1, x_3) \text{ and } y_2 = \Phi(x_1, x_2) = \Phi(x_2, x_3).
\end{equation}

\begin{example}\label{ex3.20}
Let \((X, d)\) be a nonempty ultrametric space. Recall that a subset \(B\) of \(X\) is a (closed) ball if there are \(x^{*} \in X\) and \(r^{*} \in \RR^{+}\) such that
\[
B = \{x \in X \colon d(x, x^{*}) \leqslant r^{*}\}.
\]
The diameter of \(B\), we denote it by \(\diam(B)\), is defined as 
\[
\diam(B) := \sup\{d(x,y) \colon x, y \in B\}.
\]
The following statements are equivalent for every \(\<r_1, r_2> \in \RR^{+} \times \RR^{+}\).
\begin{itemize}
\item \(\<r_1, r_2> \in u_d\).
\item There are some balls \(B_1\) and \(B_2\) in \((X, d)\) such that \(B_1 \subseteq B_2\), and \(r_1 = \diam(B_1)\), and \(r_2 = \diam(B_2)\).
\item There are some balls \(B_1\) and \(B_2\) in \((X, d)\) such that \(B_1 \cap B_2 \neq \varnothing\), and \(r_1 = \diam(B_1)\), \(r_2 = \diam(B_2)\), and \(r_1 \leqslant r_2\).
\end{itemize}
The interchangeability of these conditions is easy to justify using the known properties of balls in ultrametric spaces (see, for example, Proposition~1.2 and Proposition~1.6 in~\cite{Dov2019PNUAA}).
\end{example}

\begin{theorem}\label{t3.7}
Let \(X\) be a nonempty set and let \(\Phi\) be a mapping with \(\dom \Phi = X^{2}\) and \(|\Phi(X^{2})| \leqslant \aleph_{0}\). Then the following conditions are equivalent.
\begin{enumerate}
\item[\((i)\)] \(\Phi\) is combinatorially similar to a pseudoultrametric \(d \colon X^{2} \to \RR^{+}\) with \(d(X^{2}) \subseteq \QQ^{+}\).
\item[\((ii)\)] \(\Phi\) is combinatorially similar to a pseudoultrametric.
\item[\((iii)\)] The mapping \(\Phi\) is symmetric, and the transitive closure \(u_{\Phi}^{t}\) of the binary relation \(u_{\Phi}\) is antisymmetric, and \(\Phi\) is \(a_0\)-coherent for a point \(a_0 \in \Phi(X^{2})\), and, for every triple \(\<x_1, x_2, x_3>\) of points of \(X\), there is a permutation
\[
\begin{pmatrix}
x_1 & x_2 & x_3\\
x_{i_1} & x_{i_2} & x_{i_3}
\end{pmatrix}
\]
such that \(\Phi(x_{i_1}, x_{i_2}) = \Phi(x_{i_2}, x_{i_3})\).
\end{enumerate}
\end{theorem}

\begin{proof}
\((i) \Rightarrow (ii)\). This is trivially valid.

\((ii) \Rightarrow (iii)\). Suppose \(\Phi\) is combinatorially similar to a pseudoultrametric. Then \(\Phi\) also is combinatorially similar to a pseudometric. Consequently, by Theorem~\ref{ch2:p7}, \(\Phi\) is symmetric and there is \(a_0 \in \Phi(X^{2})\) such that \(\Phi\) is \(a_0\)-coherent. If \(\<x_1, x_2, x_3>\) is an arbitrary triple of points of \(X\), then, by Lemma~\ref{l3.1}, there is a permutation 
\[
\begin{pmatrix}
x_1 & x_2 & x_3\\
x_{i_1} & x_{i_2} & x_{i_3}
\end{pmatrix}
\]
such that \(\Phi(x_{i_1}, x_{i_2}) = \Phi(x_{i_2}, x_{i_3})\). To complete the proof of validity of \((ii) \Rightarrow (iii)\) it suffices to show that the transitive closure \(u_{\Phi}^{t}\) of the binary relation 
\[
u_{\Phi} \subseteq Y^{2}, \quad Y = \Phi(X^{2}),
\]
is antisymmetric. Suppose contrary that there are distinct \(y_1\), \(y_2 \in Y\) such that \(\<y_1, y_2> \in u_{\Phi}^{t}\) and \(\<y_2, y_1> \in u_{\Phi}^{t}\). The definition of the transitive closure (see \eqref{e3.2}) and the definition of the composition of binary relations imply that there are a positive integer \(n_1\) and some points
\[
y_{1}^{*},\ y_{2}^{*},\ \ldots,\ y_{n_1+1}^{*} \in Y 
\]
with
\begin{equation}\label{t3.7:e2}
y_{1}^{*} = y_1 \quad \text{and} \quad y_{n_1+1}^{*} = y_2 \quad \text{and} \quad \<y_{i}^{*}, y_{i+1}^{*}> \in u_{\Phi}
\end{equation}
for \(i = 1\), \(\ldots\), \(n_1\). Since \(\Phi\) is combinatorially similar to a pseudoultrametric \(d \colon Z^2 \to \RR^{+}\), there are bijections
\[
g \colon Z \to Y \text{ and } f \colon \Phi(X^{2}) \to d(Z^{2})
\]
satisfying
\[
d(z_1, z_2) = f(\Phi(g(z_1), g(z_2)))
\]
for all \(z_1\), \(z_2 \in Z\). Consequently,
\[
d(g^{-1}(x_1), g^{-1}(x_2)) = f(\Phi(x_1, x_2))
\]
holds for all \(x_1\), \(x_2 \in X\). As in Example~\ref{ex3.4}, the last equality, \eqref{e2.16}, \eqref{t3.7:e2}, and the strong triangle inequality imply
\[
f(y_1) = f(y_{1}^{*}) \geqslant f(y_{2}^{*}) \geqslant \ldots \geqslant f(y_{n_1+1}^{*}) = f(y_{2}).
\]
Thus, the inequality \(f(y_1) \geqslant f(y_{2})\) holds. Similarly, we can obtain the inequality \(f(y_2) \geqslant f(y_{1})\). 

Consequently, the equality \(f(y_1) = f(y_{2})\) holds, that contradicts the bijectivity of \(f\).

\((iii) \Rightarrow (i)\). Suppose \(\Phi\) satisfies condition \((iii)\). Let us define a binary relation \({\preccurlyeq}\) on \(Y = \Phi(X^{2})\) as
\begin{equation}\label{t3.7:e3}
{\preccurlyeq} := u_{\Phi}^{t} \cup \Delta_{Y},
\end{equation}
where \(\Delta_{Y} = \{\<y, y> \colon y \in Y\}\). We claim that \({\preccurlyeq}\) is a partial order on \(Y\). Indeed, \eqref{t3.7:e3} implies that \({\preccurlyeq}\) is reflexive. By condition~\((iii)\), the transitive closure \(u_{\Phi}^{t}\) is antisymmetric. From this and \eqref{t3.7:e3} it follows that \({\preccurlyeq}\) is also antisymmetric. Moreover, using the transitivity of \(u_{\Phi}^{t}\) we obtain
\begin{align*}
(u_{\Phi}^{t} \cup \Delta_{Y})^2 &= (u_{\Phi}^{t} \circ u_{\Phi}^{t}) \cup (u_{\Phi}^{t} \circ \Delta_{Y}) \cup (\Delta_{Y} \circ u_{\Phi}^{t}) \cup (\Delta_{Y} \circ \Delta_{Y}) \\
& \subseteq u_{\Phi}^{t} \cup \Delta_{Y}.
\end{align*}
Consequently, \({\preccurlyeq}\) is transitive. Thus, \({\preccurlyeq}\) is a partial order as required. 

By condition~\((iii)\), \(\Phi\) is \(a_0\)-coherent. We will show that \(a_0\) is the smallest element of the poset \((Y, {\preccurlyeq})\). 

Let \(y_1\) be an arbitrary point of \(Y\). Then there are \(x_1\), \(x_2 \in X\) such that \(y_1 = \Phi(x_1, x_2)\). The mapping \(\Phi\) is symmetric. Thus, 
\begin{equation}\label{t3.7:e4}
\Phi(x_1, x_2) = \Phi(x_2, x_1)
\end{equation}
holds. Since \(\Phi\) is \(a_0\)-coherent, we have 
\begin{equation}\label{t3.7:e5}
\Phi(x_1, x_1) = a_0.
\end{equation}
Using~\eqref{t3.7:e4}, \eqref{t3.7:e5} and the definition of \(u_{\Phi}\) we obtain \(\<a_0, y_1> \in u_{\Phi}\) for every \(y_1 \in Y\), as required. 

Write \(\preccurlyeq_{0}\) for the intersection \({\preccurlyeq}\) with the set \(Y_0^{2}\), where 
\[
Y_0 = \{y \in Y \colon y \neq a_0\}.
\]
Then \(\preccurlyeq_{0}\) is a partial order on the set \(Y_0\). By Lemma~\ref{l3.3}, there is a linear order \(\preccurlyeq^{*}\) on \(Y_0\) such that 
\[
{\preccurlyeq_{0}} \subseteq {\preccurlyeq^{*}}.
\]
The inequality \(|Y| \leqslant \aleph_{0}\) implies \(|Y_0| \leqslant \aleph_{0}\). Using Lemma~\ref{l3.2} we can find an injective mapping \(f^{*} \colon Y \to \mathbb{Q}^{+}\) such that \(f^{*}(a_0) = 0\) and
\[
(y_1 \preccurlyeq^{*} y_2) \Leftrightarrow (f^{*}(y_1) \leqslant f^{*}(y_2))
\]
for all \(y_1\), \(y_2 \in Y\). Then the function \(d \colon X^{2} \to \RR^{+}\),
\[
d(x_1, x_2) = f^{*}(\Phi(x_1, x_2)), \quad x_1, x_2 \in X,
\]
is a pseudoultrametric on \(X\) and \(d(X^{2}) \subseteq \QQ^{+}\) holds. Since the function \(f^{*}\) is injective, the identical mapping \(X \xrightarrow{\operatorname{id}} X\) is a combinatorial similarity.
\end{proof}

Using Theorem~\ref{t3.7} and Corollary~\ref{c3.2} we also obtain.

\begin{corollary}\label{c3.8}
Let \(X\) be a nonempty set. The following conditions are equivalent for every mapping \(\Phi\) with \(\dom \Phi = X^{2}\) and \(|\Phi(X^{2})| \leqslant \aleph_{0}\).
\begin{enumerate}
\item[\((i)\)] \(\Phi\) is combinatorially similar to an ultrametric \(d \colon X^{2} \to \RR^{+}\) satisfying the inclusion \(d(X^{2}) \subseteq \QQ^{+}\).
\item[\((ii)\)] \(\Phi\) is combinatorially similar to an ultrametric.
\item[\((iii)\)] \(\Phi\) is symmetric, and the transitive closure \(u_{\Phi}^{t}\) of the binary relation \(u_{\Phi}\) is antisymmetric, and the equality
\[
\Phi^{-1}(a_0) = \Delta_{X}
\]
holds for some \(a_0 \in \Phi(X^{2})\), and, for every triple \(\<x_1, x_2, x_3>\) of points of \(X\), there is a permutation
\[
\begin{pmatrix}
x_1 & x_2 & x_3\\
x_{i_1} & x_{i_2} & x_{i_3}
\end{pmatrix}
\]
such that \(\Phi(x_{i_1}, x_{i_2}) = \Phi(x_{i_2}, x_{i_3})\).
\end{enumerate}
\end{corollary}

\begin{example}\label{ex3.9}
A four-point metric space \((X, d)\) is called a \emph{pseudolinear quadruple} (see~\cite{Blu1953CP} for instance) if, for a suitable enumeration of points of \(X\), we have
\begin{multline}\label{ex3.9:e1}
d(x_1, x_2) = d(x_3, x_4) = s, \quad d(x_2, x_3) = d(x_4, x_1) = t, \\
d(x_2, x_4) = d(x_3, x_1) = s + t,
\end{multline}
with some positive reals \(s\) and \(t\). For a pseudolinear quadruple \((X, d)\), Corollary~\ref{c3.8} implies that the metric \(d \colon X^{2} \to \RR^{+}\) is combinatorially similar to an ultrametric if and only if \((X, d)\) is ``equilateral'', i.e., \eqref{ex3.9:e1} holds with \(s= t\) (see Figure~\ref{fig2}). 
\begin{figure}[htb]
\begin{tikzpicture}[scale=1]
\draw (0,0) circle (1cm);
\draw [fill=black] (-1,0) circle (2pt) node [label=left:\(x_1\)] {};
\draw [fill=black] (0,-1) circle (2pt) node [label=below:\(x_2\)] {};
\draw [fill=black] (1,0) circle (2pt) node [label=right:\(x_3\)] {};
\draw [fill=black] (0,1) circle (2pt) node [label=above:\(x_4\)] {};
\end{tikzpicture}
\caption{Each equilateral, pseudolinear quadruple is (up to similarity) a subspace \(\{x_1, x_2, x_3, x_4\}\) of the unit circle endowed with the shortest path metric.}
\label{fig2}
\end{figure}
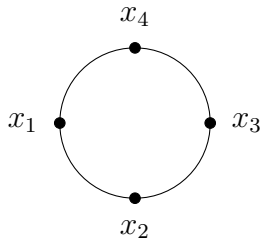
\end{example}

\begin{remark}\label{r3.10}
The pseudolinear quadruples appeared for the first time in the paper of Menger~\cite{Men1928MA}. According to Menger, the pseudolinear quadruples are characterized as the metric spaces which are not isometric to any subset of \(\RR\), but such that every triple of whose points embeds isometrically into \(\RR\). There is also an elementary proof of this fact \cite{DD2009UMZ}. It is interesting to note that the equilateral, pseudolinear quadruples are the ``most non-Ptolemaic'' metric spaces~\cite{DP2011SMJ}.
\end{remark}

For what follows we need a specification of the concept of ultrametric distances introduced above in Definition~\ref{d1.3}.

\begin{definition}\label{d3.11}
Let \((Q, \preccurlyeq_Q)\) be a poset with a smallest element \(q_0\) and let \(X\) be a nonempty set. A mapping \(d \colon X^2 \to Q\) is a \(\preccurlyeq_Q\)-\emph{pseudo\-ultra\-metric} if \(d\) is symmetric and \(d(x, x) = q_0\) holds for every \(x \in X\) and, in addition, for every triple \(\<x_1, x_2, x_3>\) of points of \(X\), there is a permutation
\[
\begin{pmatrix}
x_1 & x_2 & x_3\\
x_{i_1} & x_{i_2} & x_{i_3}
\end{pmatrix}
\]
such that
\begin{equation}\label{d3.11:e1}
d(x_{i_1}, x_{i_3}) \preccurlyeq_Q d(x_{i_1}, x_{i_2}) \quad \text{and} \quad  d(x_{i_1}, x_{i_2}) = d(x_{i_2}, x_{i_3}).
\end{equation}
For \(\preccurlyeq_{Q}\)-pseudoultrametric \(d\), satisfying \(d(x, y) = q_0\) if and only if \(x = y\), we say that \(d\) is a \(\preccurlyeq_{Q}\)-\emph{ultrametric}.
\end{definition}

If there is no ambiguity in the choice of the order \(\preccurlyeq_{Q}\) we write ``\(d\) is a \(Q\)-pseudoultrametric'' instead of ``\(d\) is a \(\preccurlyeq_Q\)-pseudoultrametric''.

\begin{remark}\label{r3.13}
It is easy to prove that every ultrametric is a \(\leqslant\)-ultra\-metric for \((\RR^{+}, \leqslant)\). Moreover, every \({\preccurlyeq}_{Q}\)-ultrametric is an ultrametric distance with the same \((Q, {\preccurlyeq}_Q)\) but not conversely (see, in particular, Example~\ref{ex3.26} at the end of the present section). For all totally ordered sets \(Q\), the ultrametric distances coincide with \(Q\)-ultrametrics, and with generalized ultrametrics defined by Priess-Crampe \cite{Pri1990RiM}.
\end{remark}

The following proposition is an extension of Proposition~\ref{c2.2} for the case of arbitrary \(Q\)-pseudoultrametric.

\begin{proposition}\label{p3.12}
Let \(X\) be a nonempty set and \((Q, \preccurlyeq_Q)\) be a poset with the smallest element \(q_0\) and let \(d \colon X^2 \to Q\) be a \(Q\)-pseudo\-ultra\-metric on \(X\). Then \(d^{-1}(q_0)\) is an equivalence relation on \(X\) and the mapping \(d\) is \(q_0\)-coherent.
\end{proposition}

\begin{proof}
It follows directly from Definition~\ref{d3.11} that \(d^{-1}(q_0)\) is reflexive. To prove that \(d^{-1}(q_0)\) is symmetric it suffices to note that the mapping \(d \colon X^{2} \to Q\) is symmetric, because, for each mapping \(\Phi\) with
\[
\dom \Phi = X^{2},
\]
\(\Phi\) is symmetric if and only if \(\Phi^{-1}(b)\) is a symmetric binary relation for every \(b \in \Phi(X^{2})\). Thus, \(d^{-1}(q_0)\) is an equivalence relation if and only if \(d^{-1}(q_0)\) is transitive.

Let \(\<x_1, x_2>\) and \(\<x_2, x_3>\) belong to \(X^{2}\) and let 
\begin{equation}\label{p3.12:e1}
d(x_1, x_2) = d(x_2, x_3) = q_0.
\end{equation}
We claim that \(d(x_1, x_3) = q_0\) holds. Indeed, by Definition~\ref{d3.11}, there is a permutation
\[
\begin{pmatrix}
x_1 & x_2 & x_3\\
x_{i_1} & x_{i_2} & x_{i_3}
\end{pmatrix}
\]
such that~\eqref{d3.11:e1} holds. From~\eqref{p3.12:e1} and \eqref{d3.11:e1} it follows that 
\begin{equation}\label{p3.12:e2}
d(x_{i_1}, x_{i_2}) = d(x_{i_2}, x_{i_3}) = q_0.
\end{equation}
Using \eqref{d3.11:e1} again we see that \eqref{p3.12:e2} implies
\begin{equation}\label{p3.12:e3}
d(x_{i_1}, x_{i_3}) \preccurlyeq_Q q_0.
\end{equation}
Since \(q_0\) is the smallest element of \((Q, \preccurlyeq_Q)\), inequality \eqref{p3.12:e3} implies 
\begin{equation}\label{p3.12:e4}
d(x_{i_1}, x_{i_3}) = q_0.
\end{equation}
The equality \(d(x_1, x_3) = q_0\) follows from \eqref{p3.12:e4} and \eqref{p3.12:e2}. Thus, \(d^{-1}(q_0)\) is transitive.

Now we need to prove that \(d\) is \(q_0\)-coherent. The mapping \(d\) is symmetric. Hence, by Corollary~\ref{c2.5}, it suffices to show that 
\begin{equation}\label{p3.12:e5}
d^{-1}(q) = d^{-1}(q_1) \circ d^{-1}(q_0)
\end{equation}
for every \(q_1 \in d(X^{2})\). Let \(q_1 \in d(X^{2})\). We have
\[
d^{-1}(q_1) \subseteq d^{-1}(q_1) \circ d^{-1}(q_0),
\]
because \(d^{-1}(q_0)\) is reflexive. The converse inclusion
\begin{equation}\label{p3.12:e6}
d^{-1}(q_1) \supseteq d^{-1}(q_1) \circ d^{-1}(q_0)
\end{equation}
holds if and only if, for all \(x_1\), \(x_2\), \(x_3 \in X\), we have
\begin{equation}\label{p3.12:e7}
\<x_1, x_3> \in d^{-1}(q_1)
\end{equation}
whenever \(\<x_1, x_2> \in d^{-1}(q_1)\) and \(\<x_2, x_3> \in d^{-1}(q_0)\). If \(q_1 = q_0\), then \eqref{p3.12:e6} holds, since \(d^{-1}(q_0)\) is an equivalence relation. Suppose 
\[
q_1 \neq q_0.
\]
Write \(q_2 := d(x_1, x_3)\). If \(q_2 = q_1\), then \eqref{p3.12:e7} follows from \(\<x_1, x_3> \in d^{-1}(q_2)\). Consequently, if \eqref{p3.12:e7} is false, then we have 
\begin{equation}\label{p3.12:e8}
q_2 \neq q_1 \neq q_0.
\end{equation}
The equality \(q_2 = q_0\) implies
\begin{equation}\label{p3.12:e9}
\<x_1, x_3> \in d^{-1}(q_0),
\end{equation}
because \(d^{-1}(q_0)\) is transitive. From \eqref{p3.12:e9} and \eqref{p3.12:e7} follows \(q_0 = q_1\), contrary to \eqref{p3.12:e8}. Thus, \(q_0\), \(q_1\) and \(q_2\) are pairwise distinct, that contradicts~\eqref{d3.11:e1}.
\end{proof}

\begin{corollary}\label{c3.17}
Let \(X\) be a nonempty set and \((Q, \preccurlyeq_Q)\) be a poset and let \(d \colon X^{2} \to Q\) be a \(Q\)-pseudoultrametric (\(Q\)-ultrametric) on \(X\). Then the following statements are valid.
\begin{enumerate}
\item [\((i)\)] If \(|d(X^2)| \leqslant 2^{\aleph_{0}}\) holds, then \(d\) is combinatorially similar to an usual pseudometric (metric).
\item [\((ii)\)] If \(|d(X^2)| \leqslant \aleph_{0}\) holds, then \(d\) is combinatorially similar to an usual pseudoultrametric (ultrametric).
\end{enumerate}
\end{corollary}

\begin{proof}
Suppose first that \(d\) is a \(Q\)-pseudoultrametric. 

\((i)\). If \(|d(X^2)| \leqslant 2^{\aleph_{0}}\) holds, then Definition~\ref{d3.11} and Proposition~\ref{p3.12} imply condition \((ii)\) of Theorem~\ref{ch2:p7}. Thus, \((i)\) is valid by Theorem~\ref{ch2:p7}.

\((ii)\). Analogously, using Definition~\ref{d3.11} we can show that condition \((iii)\) of Theorem~\ref{t3.7} is valid for \(\Phi = d\). Thus, \((ii)\) follows from Theorem~\ref{t3.7}.

The case when \(d\) is a \(Q\)-ultrametric can be considered similarly.
\end{proof}

The next theorem is a partial generalization of Theorem~\ref{t3.7}.

\begin{theorem}\label{t3.15}
Let \(X\) be a nonempty set and let \(\Phi\) be a mapping with \(\dom \Phi = X^{2}\). Then the following conditions are equivalent.
\begin{enumerate}
\item[\((i)\)] There is a totally ordered set \(Q\) such that \(\Phi\) is combinatorially similar to a \(Q\)-pseudoultrametric.
\item[\((ii)\)] There is a poset \(Q\) such that \(\Phi\) is combinatorially similar to a \(Q\)-pseudoultrametric.
\item[\((iii)\)] The mapping \(\Phi\) is symmetric, and the transitive closure \(u_{\Phi}^{t}\) of the binary relation \(u_{\Phi}\) is antisymmetric, and there is \(a_0 \in \Phi(X^{2})\) for which \(\Phi\) is \(a_0\)-coherent, and, for every triple \(\<x_1, x_2, x_3>\) of points of \(X\), there is a permutation
\[
\begin{pmatrix}
x_1 & x_2 & x_3\\
x_{i_1} & x_{i_2} & x_{i_3}
\end{pmatrix}
\]
such that \(\Phi(x_{i_1}, x_{i_2}) = \Phi(x_{i_2}, x_{i_3})\).
\item[\((iv)\)] There is \(b_0 \in \Phi(X^{2})\) such that \(\Phi(x,x) = b_0\) holds for every \(x \in X\), and the binary relation 
\begin{equation}\label{t3.15:e1}
{\preccurlyeq}_{\Phi} := u_{\Phi}^{t} \cup \Delta_{\Phi(X^{2})}
\end{equation}
is a partial order on \(\Phi(X^{2})\), and \(b_0\) is the smallest element of \((\Phi(X^{2}), {\preccurlyeq}_{\Phi})\), and \(\Phi\) is a \( {\preccurlyeq}_{\Phi}\)-pseudo\-ultra\-metric on \(X\).
\end{enumerate}
\end{theorem}

\begin{proof}
The implication \((i) \Rightarrow (ii)\) is trivially valid. The validity of \((ii) \Rightarrow (iii)\) can be verified by repetition of the first part of the proof of Theorem~\ref{t3.7} with the replacement of the word ``Theorem~\ref{ch2:p7}'' by word ``Proposition~\ref{p3.12}''. It should be noted that Lemma~\ref{l3.1} remains valid if \(\Phi\) is combinatorially similar to an arbitrary \(Q\)-pseudoultrametric.

\((iii) \Rightarrow (iv)\). Let \((iii)\) hold. Then \(u_{\Phi}^{t}\) is antisymmetric and transitive. Consequently, the relation \({\preccurlyeq}_{\Phi}\) is reflexive, antisymmetric and transitive, i.e., \({\preccurlyeq}_{\Phi}\) is a partial order on \(\Phi(X^{2})\). Since \(\Phi\) is \(a_0\)-coherent, the equality \(\Phi(x,x) = a_0\) holds for every \(x \in X\).

The point \(a_0\) is the smallest element of \((\Phi(X^{2}), {\preccurlyeq}_{\Phi})\) if and only if the inequality
\begin{equation}\label{t3.15:e2}
a_0 \preccurlyeq_{\Phi} \Phi(x, y)
\end{equation}
holds for all \(x\), \(y \in X\). To prove \eqref{t3.15:e2} we consider the triple \(\<y, x, y>\) and note that \(\Phi(x, y) = \Phi(y,x)\). Consequently, \(\<\Phi(x,x), \Phi(x,y)>\) belongs to \(u_{\Phi}\). Now \eqref{t3.15:e2} follows from \eqref{t3.15:e1}.

By condition \((iii)\), \(\Phi(x,x) = a_0\) holds for every \(x \in X\) and, for every triple \(\<x_1, x_2, x_3>\) of points of \(X\), there is a permutation
\[
\begin{pmatrix}
x_1 & x_2 & x_3\\
x_{i_1} & x_{i_2} & x_{i_3}
\end{pmatrix}
\]
such that \(\Phi(x_{i_1}, x_{i_2}) = \Phi(x_{i_2}, x_{i_3})\). The mapping \(\Phi\) is symmetric.  Hence, \(\Phi\) is a \({\preccurlyeq}_{\Phi}\)-pseudoultrametric on \(X\) as required.

\((iv) \Rightarrow (i)\). Let \((iv)\) hold. Then \(\Phi\) is a \({\preccurlyeq}_{\Phi}\)-pseudoultrametric. By Lemma~\ref{l3.3} (Szpilrajn) the partial order \({\preccurlyeq}_{\Phi}\) can be extended to a linear order \({\preccurlyeq}\) on \(\Phi(X^{2})\). It is easy to see that the smallest element \(a_0\) of \((\Phi(X^{2}), {\preccurlyeq}_{\Phi})\) is also the smallest element of \((\Phi(X^{2}), {\preccurlyeq})\). Thus, \(\Phi\) is also a \({\preccurlyeq}\)-pseudoultrametric. Condition \((i)\) follows.
\end{proof}

\begin{corollary}\label{c3.19}
Let \(X\) be a nonempty set and let \(\Phi\) be a mapping with \(\dom \Phi = X^{2}\). Then the following conditions are equivalent.
\begin{enumerate}
\item[\((i)\)] There is a totally ordered set \(Q\) such that \(\Phi\) is combinatorially similar to a \(Q\)-ultra\-metric.
\item[\((ii)\)] There is a poset \(Q\) such that \(\Phi\) is combinatorially similar to a \(Q\)-ultra\-metric.
\item[\((iii)\)] The mapping \(\Phi\) is symmetric, and the transitive closure \(u_{\Phi}^{t}\) of the binary relation \(u_{\Phi}\) is antisymmetric, and there is \(a_0 \in \Phi(X^{2})\) for which \(\Phi^{-1}(a_0) = \Delta_{X}\) holds, and, for every triple \(\<x_1, x_2, x_3>\) of points of \(X\), there is a permutation
\[
\begin{pmatrix}
x_1 & x_2 & x_3\\
x_{i_1} & x_{i_2} & x_{i_3}
\end{pmatrix}
\]
such that \(\Phi(x_{i_1}, x_{i_2}) = \Phi(x_{i_2}, x_{i_3})\).
\item[\((iv)\)] There is \(b_0 \in \Phi(X^{2})\) such that \(\Phi^{-1}(b_0) = \Delta_{X}\) holds, and the binary relation 
\[
{\preccurlyeq}_{\Phi} := u_{\Phi}^{t} \cup \Delta_{\Phi(X^{2})}
\]
is a partial order on \(\Phi(X^{2})\), and \(b_0\) is the smallest element of \((\Phi(X^{2}), {\preccurlyeq}_{\Phi})\), and \(\Phi\) is a \( {\preccurlyeq}_{\Phi}\)-ultra\-metric on \(X\).
\end{enumerate}
\end{corollary}

The next corollary follows from Corollary~\ref{c3.8} and Corollary~\ref{c3.19}.

\begin{corollary}\label{c3.20}
Let \((Q, {\preccurlyeq}_{Q})\) be a poset with a smallest element, let \(X\) be a nonempty set and let \(d \colon X^{2} \to Q\) be an ultrametric distance in the sense of Priess-Crampe and Ribenboim.  If the inequality \(|Q| \leqslant \aleph_{0}\) holds, then the following conditions are equivalent.
\begin{enumerate}
\item [\((i)\)] The mapping \(d\) is a \(Q\)-ultra\-metric.
\item [\((ii)\)] There is an usual ultrametric \(\rho \colon X^{2} \to \RR^{+}\) such that \(d\) and \(\rho\) are combinatorially similar.
\end{enumerate}
\end{corollary}

The following proposition guarantees, for a given \(Q\)-pseudo\-ultra\-metric \(d\), the presence of the weakest (on \(Q\)) partial order at which \(d\) remains \(Q\)-pseudo\-ultra\-metric.

\begin{proposition}\label{p3.17}
Let \(X\) be a nonempty set, \((Q, {\preccurlyeq}_{Q})\) be a poset and let \(d \colon X^{2} \to Q\) be a \({\preccurlyeq}_{Q}\)-pseudoultrametric. Then there is a unique partial order \({\preccurlyeq}_{Q}^{0}\) on \(Q\) such that \(d\) is a \({\preccurlyeq}_{Q}^{0}\)-pseudoultrametric and the inclusion
\[
{\preccurlyeq}_{Q}^{0} \subseteq {\preccurlyeq}
\]
holds whenever \({\preccurlyeq}\) is a partial order on \(Q\) for which \(d\) is a \({\preccurlyeq}\)-pseudo\-ultrametric.
\end{proposition}

\begin{proof}
The uniqueness of \({\preccurlyeq}_{Q}^{0}\) satisfying the desirable conditions is clear. For the proof of existence of \({\preccurlyeq}_{Q}^{0}\), let \(\mathcal{F} = \{{\preccurlyeq}_i \colon i \in I\}\) be the family of all partial orders \({\preccurlyeq}_i\) on \(Q\) for which \(d\) is a \({\preccurlyeq}_i\)-pseudoultrametric. The family \(\mathcal{F}\) is non-void  because \({\preccurlyeq}_{Q} \in \mathcal{F}\). Let us define a binary relation \({\preccurlyeq}_{Q}^{0}\) as the intersection of all \({\preccurlyeq}_i\), i.e., for \(p\), \(q \in Q\),
\[
(\<p,q> \in {\preccurlyeq}_{Q}^{0}) \Leftrightarrow (p \preccurlyeq_i q \text{ holds for every } i \in I).
\]
Then \({\preccurlyeq}_{Q}^{0}\) is a partial order on \(Q\). Since \(d\) is a \({\preccurlyeq}_{Q}\)-pseudoultrametric, the poset \((Q, {\preccurlyeq}_{Q})\) has a smallest element \(q_0\) by definition. It is easy to prove that \(q_0\) is the common smallest element of all posets \((Q, {\preccurlyeq}_i)\), \(i \in I\).

Indeed, since \(d\) is a \({\preccurlyeq}_{Q}\)-pseudoultrametric, we have \(d(x, x) = q_0\). In addition, since, for arbitrary \(i^{*} \in I\), the mapping \(d\) is a \({\preccurlyeq}_{i^*}\)-pseudo\-ultrametric, we also have 
\[
d(x, x) = q_0^{*},
\]
where \(q_0^{*}\) is the smallest element of \((Q, {\preccurlyeq}_{i^*})\). That implies \(q_0^{*} = q_0\). 

Consequently, \(q_0\) is the smallest element of \((Q, {\preccurlyeq}_{Q}^{0})\). 

Hence, to prove that \(d\) is a \({\preccurlyeq}_{Q}^{0}\)-pseudoultrametric it suffices to show that for every triple \(\<x_1, x_2, x_3>\) of points of \(X\) there is a permutation
\[
\begin{pmatrix}
x_1 & x_2 & x_3\\
x_{i_1} & x_{i_2} & x_{i_3}
\end{pmatrix}
\]
such that
\begin{equation}\label{p3.17:e0}
d(x_{i_1}, x_{i_3}) \preccurlyeq_{Q}^{0} d(x_{i_1}, x_{i_2}) \text{ and } d(x_{i_1}, x_{i_2}) = d(x_{i_2}, x_{i_3}).
\end{equation}
Condition~\eqref{p3.17:e0} evidently holds if 
\begin{equation}\label{p3.17:e1}
d(x_1, x_2) = d(x_2, x_3) = d(x_3, x_1).
\end{equation}
If \eqref{p3.17:e1} does not hold, then we may set, for definiteness, that
\begin{equation}\label{p3.17:e2}
d(x_1, x_2) = d(x_2, x_3) \neq d(x_1, x_3).
\end{equation}
(The case when \(d(x_1, x_2)\), \(d(x_2, x_3)\) and \(d(x_1, x_3)\) are pairwise distinct is impossible because \(d\) is a \({\preccurlyeq}_Q\)-pseudoultrametric.) Using \eqref{p3.17:e2} and \eqref{d3.11:e1} we obtain
\begin{equation}\label{p3.17:e3}
d(x_1, x_3) \preccurlyeq_i d(x_1, x_2) \text{ and } d(x_1, x_2) = d(x_2, x_3)
\end{equation}
for every \(i \in I\), that, together with the equality
\[
{\preccurlyeq}_{Q}^{0} = \bigcap_{i \in I} {\preccurlyeq}_{i},
\]
implies
\[
d(x_1, x_3) \preccurlyeq_{Q}^{0} d(x_1, x_2) \text{ and } d(x_1, x_2) = d(x_2, x_3). \qedhere
\]
\end{proof}

\begin{lemma}\label{l3.18}
Let \(X\) be a nonempty set, \((Q, {\preccurlyeq}_{Q})\) be a poset and let \(d \colon X^{2} \to Q\) be a \({\preccurlyeq}_{Q}\)-pseudo\-ultrametric with \(d(X^{2}) = Q\). Then the equality
\begin{equation}\label{l3.18:e1}
{\preccurlyeq}_{Q}^{0} = (u_d^t \cup \Delta_{Q})
\end{equation}
holds, where \(\Delta_{Q} := \{\<q,q> \colon q \in Q\}\).
\end{lemma}

\begin{proof}
As in the second part of the proof of Theorem~\ref{t3.7} we see that \(u_{d}^{t} \cup \Delta_{Q}\) is reflexive and transitive. Using \({\preccurlyeq}_{Q}\) instead of \(\leqslant\) and arguing as in the first part of that proof we obtain the antisymmetry of \(u_{d}^{t} \cup \Delta_{Q}\). Consequently, \(u_{d}^{t} \cup \Delta_{Q}\) is a partial order on \(Q\).

Let \({\preccurlyeq}\) be an arbitrary partial order on \(Q\) for which \(d\) is a \({\preccurlyeq}\)-pseudo\-ultrametric. Then, using Definition~\ref{d3.11} and the definition of \(u_d\), we see that
\[
u_{d} \subseteq {\preccurlyeq}.
\]
The last inclusion implies 
\[
(u_{d}^{t} \cup \Delta_{Q}) \subseteq ({\preccurlyeq}^{t} \cup \Delta_{Q}) = {\preccurlyeq}.
\]
Consequently, \({\preccurlyeq}_{Q}^{0} \supseteq (u_{d}^{t} \cup \Delta_{Q})\) holds.

From the definition of the relation \(u_{d}\), Definition~\ref{d3.11} and the fact that \(d\) is \({\preccurlyeq}_{Q}\)-pseudo\-ultrametric it follows that \(d\) is a \((u_{d}^{t} \cup \Delta_{Q})\)-pseudo\-ultra\-metric. Thus, equality~\eqref{l3.18:e1} holds.
\end{proof}

\begin{remark}\label{r3.18}
Equality~\eqref{l3.18:e1} does not hold if \(d(X^{2}) \neq Q\). Indeed, if \(q_1 \in Q \setminus d(X^{2})\), then we evidently have \(q_1 \notin u_d^t\), that implies
\[
\<q, q_1> \notin (u_d^t \cup \Delta_{Q})
\]
for every \(q \in Q \setminus \{q_1\}\). Consequently, the poset \((Q, u_d^t \cup \Delta_{Q})\) does not have any smallest element. The last statement contradicts \eqref{l3.18:e1}, because the smallest element \(q_0 \in d(X^{2})\) of \((Q, {\preccurlyeq}_Q)\) is also the smallest element of \((Q, {\preccurlyeq}_Q^{0})\).
\end{remark}

Results of the present section are based on the fact that, for all posets \((Q, {\preccurlyeq}_Q)\) and \((L, {\preccurlyeq}_L)\) with the smallest elements \(q_0 \in Q\) and \(l_0 \in L\), for every isotone injection \(f \colon Q \to L\) satisfying the condition \(f(q_0) = l_0\), and for each \(Q\)-pseudoultrametric \(d\), the mappings \(d\) and \(f \circ d\) are combinatorially similar. Moreover, in this case the transformation \(d \mapsto f \circ d\) converts the \(Q\)-pseudoultrametrics into \(L\)-pseudoultrametrics.

\begin{proposition}\label{p3.23}
Let \((Q, {\preccurlyeq}_Q)\) and \((L, {\preccurlyeq}_L)\) be posets with the smallest elements \(q_0 \in Q\) and \(l_0 \in L\). The following conditions are equivalent for every mapping \(f \colon Q \to L\).
\begin{enumerate}
\item [\((i)\)] \(f \circ d\) is a \(L\)-pseudoultrametric whenever \(d\) is a \(Q\)-pseudo\-ultra\-metric.
\item [\((ii)\)] \(f \circ d\) is a \(L\)-pseudoultrametric whenever \(d\) is a \(Q\)-ultra\-metric.
\item [\((iii)\)] \(f\) is isotone and \(f(q_0) = l_0\) holds.
\end{enumerate}
\end{proposition}

\begin{proof}
\((i) \Rightarrow (ii)\). This is evidently valid.

\((ii) \Rightarrow (iii)\). Suppose statement \((ii)\) is valid. Then, for every \(Q\)-ultrametric space \((X, d)\) and for every \(x \in X\), the equalities
\[
f(q_0) = f(d(x, x)) = l_0
\]
hold. Let \(q_1\), \(q_2 \in Q\) such that \(q_1 \preccurlyeq_Q q_2\). We must prove the inequality
\begin{equation}\label{p3.23:e1}
f(q_1) \preccurlyeq_L f(q_2).
\end{equation}
This is trivial if \(f(q_1) = f(q_2)\). Suppose \(f(q_1) \neq f(q_2)\) and \(X = \{x_1, x_2, x_3\}\). Let us define \(d \colon X^{2} \to L\) as
\begin{equation}\label{p3.23:e2}
d(x_1, x_2) = d(x_2, x_3) = q_2, \quad \text{and} \quad d(x_1, x_3) = q_1,
\end{equation}
and \(d(x_1, x_1) = d(x_2, x_2) = d(x_3, x_3) = q_0\). Then \(d\) is a \(Q\)-ultrametric and \(f \circ d\) is a \(L\)-pseudo\-ultra\-metric. Inequality \eqref{p3.23:e1} follows from \(f(q_1) \neq f(q_2)\), \eqref{p3.23:e2} and \eqref{d3.11:e1}.

\((iii) \Rightarrow (i)\). The validity of this implication follows directly from the definition of isotone mappings and the definition of poset-valued pseudoultrametrics.
\end{proof}

\begin{corollary}\label{c3.24}
Let \((Q, {\preccurlyeq}_Q)\) and \((L, {\preccurlyeq}_L)\) be posets with the smallest elements \(q_0 \in Q\) and \(l_0 \in L\). Then the following conditions are equivalent for every mapping \(f \colon Q \to L\).
\begin{enumerate}
\item [\((i)\)] \(f \circ d\) is a \(L\)-ultrametric whenever \(d\) is a \(Q\)-ultrametric.
\item [\((ii)\)] \(f\) is isotone and the equivalence
\begin{equation}\label{c3.24:e1}
(f(q) = l_0) \Leftrightarrow (q = q_0)
\end{equation}
is valid for every \(q \in Q\).
\end{enumerate}
\end{corollary}

\begin{proof}
\((i) \Rightarrow (ii)\). Let \((i)\) hold. Then, by Proposition~\ref{p3.23}, \(f\) is isotone and \(f(q_0) = l_0\) holds. Thus, to prove \((ii)\) it suffices to show that \(f(q) = l_0\) implies \(q = q_0\). Suppose contrary that there is \(q_1 \in Q\) such that \(q_1 \neq q_0\) and \(f(q_1) = l_0\).

Let \(X\) be an arbitrary set with \(|X| \geqslant 2\). The function \(d \colon X^{2} \to Q\), defined as
\begin{equation}\label{c3.24:e2}
d(x, y) = \begin{cases}
q_0 & \text{if } x = y,\\
q_1 & \text{if } x \neq y,
\end{cases}
\end{equation}
is a \(Q\)-ultrametric on \(X\). The equalities \(f(q_0) = l_0\), \(f(q_1) = l_0\) and \eqref{c3.24:e2} imply \(f(d(x, y)) = l_0\) for all \(x\), \(y \in X\). Hence, \(f \circ d\) is not a \(L\)-ultra\-metric on \(X\), which contradicts condition~\((i)\).

\((ii) \Rightarrow (i)\). Suppose \((ii)\) holds, but there are a set \(X\) and a \(Q\)-ultra\-metric \(d \colon X^{2} \to Q\) such that \(f \circ d\) is not a \(L\)-ultrametric. Then we evidently have \(|X| \geqslant 2\). Moreover, Proposition~\ref{p3.23} implies that \(f \circ d\) is a \(L\)-pseudoultrametric. Consequently, there are \(x_1\), \(x_2 \in X\) such that \(x_1 \neq x_2\) and 
\begin{equation}\label{c3.24:e3}
f(d(x_1, x_2)) = l_0.
\end{equation}
Since \(d\) is a \(Q\)-ultrametric,
\begin{equation}\label{c3.24:e4}
d(x_1, x_2) \neq q_0
\end{equation}
holds. From \eqref{c3.24:e3} and \eqref{c3.24:e4} it follows that \eqref{c3.24:e1} is false with \(q = d(x_1, x_2)\), contrary to condition \((ii)\).
\end{proof}

The following example shows that we cannot replace statement \((i)\) of Corollary~\ref{c3.24} by the statement 
\begin{itemize}
\item \(f \circ d\) is an ultrametric distance w.r.t \((L, {\preccurlyeq}_L)\) whenever \(d\) is an ultrametric distance w.r.t \((Q, {\preccurlyeq}_Q)\)
\end{itemize}
leaving statement \((ii)\) unchanged.

\begin{example}\label{ex3.26}
Let \(P\) and \(Q\) be sets with \(|P| = |Q| \geqslant 4\) and let \({\preccurlyeq}_{P}\) be a linear order on \(P\) with a smallest element \(p_0\). Let us define a binary relation \({\preccurlyeq}_{Q}\) on \(Q\) by the rule:
\begin{equation}\label{ex3.26:e1}
(\<q_1, q_2> \in {\preccurlyeq}_{Q}) \Leftrightarrow (q_1 = q_2 \text{ or } q_1 = q_0).
\end{equation}
Then \({\preccurlyeq}_{Q}\) is a partial order on \(Q\) and, for a set \(X = \{x_1, x_2, x_3\}\), a mapping \(d \colon X^{2} \to Q\) is an ultrametric distance w.r.t. \((Q, {\preccurlyeq}_{Q})\) if and only if \(d\) is symmetric and 
\[
(d(x, y) = q_0) \Leftrightarrow (x = y)
\]
holds for all \(x\), \(y \in X\). Since \(|Q| \geqslant 4\) holds, there is an ultrametric distance \(d^{*} \colon X^{2} \to Q\) such that \(d^{*}(x_1, x_2)\), \(d^{*}(x_2, x_3)\), \(d^{*}(x_3, x_1)\) are pairwise distinct. It follows directly from \eqref{ex3.26:e1} and Definition~\ref{d3.5} that a function \(f \colon Q \to P\) is isotone if and only if \(f(q_0) = p_0\). Now, using the equality \(|P| = |Q|\) we can find an isotone bijection \(f^{*} \colon Q \to P\) such that
\[
(f^{*}(q) = p_0) \Leftrightarrow (q = q_0)
\]
is valid for every \(q \in Q\). Since \((P, {\preccurlyeq}_{P})\) is totally ordered, and \(f^{*}\) is bijective, and \(d^{*}(x_1, x_2)\), \(d^{*}(x_2, x_3)\), \(d^{*}(x_3, x_1)\) are pairwise distinct, we can find a permutation
\[
\begin{pmatrix}
x_1 & x_2 & x_3\\
x_{i_1} & x_{i_2} & x_{i_3}
\end{pmatrix}
\]
for which
\[
f^{*}(d^{*}(x_{i_1}, x_{i_2})) \prec_{P} f^{*}(d^{*}(x_{i_2}, x_{i_3})) \prec_{P} f^{*}(d^{*}(x_{i_1}, x_{i_3})).
\]
From Definition~\ref{d1.3} it follows that the mapping
\[
X^{2} \xrightarrow{d^{*}} Q \xrightarrow{f^{*}} P
\]
is not an ultrametric distance w.r.t. \((P, {\preccurlyeq}_{P})\).
\end{example}

\begin{remark}\label{r3.25}
For the case of standard ultrametrics and pseudoultrametrics Proposition~\ref{p3.23} and Corollary~\ref{c3.24} are known. In particular, Proposition~\ref{p3.23} is a generalization of Proposition~2.4 \cite{Dov2019v2} and, respectively, Corollary~\ref{c3.24} is a generalization of Theorem~9 \cite{PTAbAppAn2014}.
\end{remark}

\section{From weak similarities to combinatorial similarities and back}

Let us expand the notion of weak similarity to the case of poset-valued pseudoultrametrics.

\begin{definition}\label{d3.13}
Let \((Q_i, {\preccurlyeq}_{Q_i})\) be a poset, and \((X_i, d_{i})\) be a \(Q_i\)-pseudo\-ultra\-metric space, and let \(Y_i := d_{i}(X_i^2)\), \(i = 1\), \(2\). A bijection \(\Phi \colon X_1 \to X_2\) is a \emph{weak similarity} for \(d_1\) and \(d_2\) if there is an isomorphism \(f \colon Y_1 \to Y_2\) of the subposet \((Y_1, {\preccurlyeq}_{Y_1})\) of the poset \((Q_1, {\preccurlyeq}_{Q_1})\) and the subposet \((Y_2, {\preccurlyeq}_{Y_2})\) of the poset \((Q_2, {\preccurlyeq}_{Q_2})\) such that
\begin{equation}\label{d3.13:e1}
d_1(x, y) = f(d_2(\Phi(x), \Phi(y)))
\end{equation}
for all \(x\), \(y \in X_1\).
\end{definition}

\begin{remark}\label{r4.2}
For every totally ordered set \((P_1, \preccurlyeq_{P_1})\) and arbitrary poset \((P_2, \preccurlyeq_{P_2})\), every isotone bijection \(f \colon P_1 \to P_2\) is an isomorphism of \((P_1, \preccurlyeq_{P_1})\) and \((P_2, \preccurlyeq_{P_2})\). Thus, Definition~\ref{d2.9} and Definition~\ref{d3.13} are equivalent for the case when \((Q_1, \preccurlyeq_{Q_1})\) and \((Q_2, \preccurlyeq_{Q_2})\) coincide with \((\RR^{+}, \leqslant)\).
\end{remark}

The following is a generalization of Proposition~\ref{p2.11}.

\begin{proposition}\label{p3.16}
Let \((Q_i, {\preccurlyeq}_{Q_i})\) be a poset and \((X_i, d_{i})\) be a \(Q_i\)-pseudoultrametric space, \(i=1\), \(2\). Then every weak similarity for \(d_1\) and \(d_2\) is a combinatorial similarity for \(d_{1}\) and \(d_{2}\).
\end{proposition}

\begin{proof}
The proposition can be directly driven from definitions. We just notice that if \(Y_1 := d_{1}(X_1^2)\) and \(Y_2 := d_{2}(X_2^2)\), and \(f \colon Y_1 \to Y_2\) is an isomorphism of the subposet \((Y_1, {\preccurlyeq}_{Y_1})\) of \((Q_1, {\preccurlyeq}_{Q_1})\) and the subposet \((Y_2, {\preccurlyeq}_{Y_2})\) of \((Q_2, {\preccurlyeq}_{Q_2})\), and~\eqref{d3.13:e1} holds for all \(x\), \(y \in X_1\), then we have \(q_2 = f(q_1)\), where \(q_i \in d_{i}(X_i^2)\) is the smallest element of \((Q_i, {\preccurlyeq}_{Q_i})\), \(i = 1\), \(2\), that agrees with Proposition~\ref{p3.12} and the second statement of Proposition~\ref{p2.4}.
\end{proof}

\begin{theorem}\label{t4.3}
Let \(X_i\) be a nonempty set and let \(\Phi_i\) be a mapping with \(\dom \Phi = X_i^{2}\), \(i = 1\), \(2\). Suppose
\begin{equation}\label{t4.3:e1}
{\preccurlyeq}_{1} := u_{\Phi_1}^{t} \cup \Delta_{\Phi_1(X_1^{2})} \quad \text{and} \quad {\preccurlyeq}_{2} := u_{\Phi_2}^{t} \cup \Delta_{\Phi_2(X_2^{2})}
\end{equation}
are partial orders on \(\Phi_1(X_1^{2})\) and, respectively, on \(\Phi_2(X_2^{2})\). If \(\Phi_i\) is a \({\preccurlyeq}_{i}\)-pseudoultrametric, \(i = 1\), \(2\), then the following conditions are equivalent for every mapping \(g \colon X_1 \to X_2\).
\begin{enumerate}
\item [\((i)\)] \(g\) is a weak similarity for \(\Phi_1\) and \(\Phi_2\).
\item [\((ii)\)] \(g\) is a combinatorial similarity for \(\Phi_1\) and \(\Phi_2\).
\end{enumerate}
\end{theorem}

\begin{proof}
Suppose \(\Phi_i\) is a \({\preccurlyeq}_{i}\)-pseudoultrametric, \(i = 1\), \(2\). 

\((i) \Rightarrow (ii)\). This is valid by Proposition~\ref{p3.16}.

\((ii) \Rightarrow (i)\). Let \(g \colon X_1 \to X_2\) be a combinatorial similarity. We must prove that \(g\) is a weak similarity for \(\Phi_1\) and \(\Phi_2\). Since \(g\) is a combinatorial similarity, there is a bijection \(f \colon \Phi_2(X_2^2) \to \Phi_1(X_1^2)\) such that
\begin{equation}\label{t4.3:e2}
\Phi_1(x, y) = f(\Phi_2(g(x), g(y)))
\end{equation}
holds for all \(x\), \(y \in X_1\). In the correspondence with Definition~\ref{d3.13}, it suffices to show that \(f\) is an isomorphism of the posets \((\Phi_1(X_1^2), {\preccurlyeq}_1)\) and \((\Phi_2(X_2^2), {\preccurlyeq}_2)\). Using \eqref{t4.3:e1} we see that if
\begin{equation}\label{t4.3:e3}
\bigl(\<a, b> \in u_{\Phi_2}\bigr) \Leftrightarrow \bigl(\<f(a), f(b)> \in u_{\Phi_1}\bigr)
\end{equation}
is valid for all \(a\), \(b \in \Phi_2(X_2^2)\), then \(f\) is an isomorphism of these posets. Condition~\eqref{t4.3:e3} follows directly from \eqref{t4.3:e2} and the definitions of \(u_{\Phi_1}\) and \(u_{\Phi_2}\).
\end{proof}

\begin{corollary}\label{c4.3}
Let \(X\) and \(Y\) be nonempty sets and let \((Q, {\preccurlyeq}_{Q})\) and \((L, {\preccurlyeq}_{L})\) be posets. Suppose \(d_Q \colon X^2 \to Q\) and \(d_L \colon Y^2 \to L\) are a \(Q\)-pseudo\-ultrametric and a \(L\)-pseudoultrametric, respectively. If \(d_{Q}(X^{2}) = Q\), and \(d_{L}(Y^{2}) = L\), and \({\preccurlyeq}_{Q} = {\preccurlyeq}_{Q}^{0}\), and \({\preccurlyeq}_{L} = {\preccurlyeq}_{L}^{0}\), then the following conditions are equivalent for every mapping \(\Phi \colon X \to Y\).
\begin{enumerate}
\item [\((i)\)] \(\Phi\) is a weak similarity for \(d_Q\) and \(d_L\).
\item [\((ii)\)] \(\Phi\) is a combinatorial similarity for \(d_Q\) and \(d_L\).
\end{enumerate}
\end{corollary}

In what follows we will use the next modification of Corollary~\ref{c4.3}.

\begin{lemma}\label{l4.7}
Let \((Q, {\preccurlyeq}_{Q})\) be a totally ordered set and let \(d \colon Q^{2} \to Q\) be a \(Q\)-pseudoultrametric such that \(d(Q^{2}) = Q\) and \({\preccurlyeq}_{Q}^{0} = {\preccurlyeq}_{Q}\). Then, for every poset \((L, {\preccurlyeq}_{L})\) having a smallest element and for each \(L\)-pseudoultrametric \(d_L \colon X^{2} \to L\) with \(d_L(X^2) = L\), the following statement holds. If \(d_L\) is combinatorially similar to \(d\), then the corresponding combinatorial similarity is a weak similarity for \(d\) and \(d_L\).
\end{lemma}

\begin{proof}
Let \((L, {\preccurlyeq}_{L})\) be a poset with a smallest element and let \(d_L\) be a pseudoultrametric on a set \(X\) with \(d_L(X^2) = L\). Suppose \(d\) and \(d_L\) are combinatorially similar. Then there are bijections 
\[
g \colon X \to Q \quad \text{and} \quad f \colon Q \to L
\]
such that the diagram
\begin{equation}\label{l4.7:e1}
\ctdiagram{
\ctv 0,25:{Q^{2}}
\ctv 100,25:{X^{2}}
\ctv 0,-25:{Q}
\ctv 100,-25:{L}
\ctet 100,25,0,25:{g\otimes g}
\ctet 0,-25,100,-25:{f}
\ctel 0,25,0,-25:{d}
\cter 100,25,100,-25:{d_L}
}
\end{equation}
is commutative. If \(f\) is an isomorphism of \((Q, {\preccurlyeq}_Q)\) and \((L, {\preccurlyeq}_L)\), then \(g\) is a weak similarity. Since \((Q, {\preccurlyeq}_Q)\) is totally ordered and \(f\) is bijective, to prove that \(f\) is an isomorphism it suffices to show that the implication
\begin{equation}\label{l4.7:e2}
(q_1 \preccurlyeq_Q q_2) \Rightarrow (f(q_1) \preccurlyeq_L f(q_2))
\end{equation}
is valid for all \(q_1\), \(q_2 \in Q\). The inclusion \({\preccurlyeq}_L^0 \subseteq {\preccurlyeq}_L\) (see Proposition~\ref{p3.17}) implies that \eqref{l4.7:e2} is valid if
\begin{equation}\label{l4.7:e3}
(q_1 \preccurlyeq_Q q_2) \Rightarrow (f(q_1) \preccurlyeq_L^0 f(q_2)).
\end{equation}
By Lemma~\ref{l3.18}, the equalities \(d(Q^2) = Q\) and \(d(X^2) = L\) imply
\begin{equation}\label{l4.7:e4}
{\preccurlyeq}_Q^0 = u_d^t \cup \Delta_{Q} \quad \text{and} \quad {\preccurlyeq}_L^0 = u_{d_L}^t \cup \Delta_{L}.
\end{equation}
Using \eqref{l4.7:e4} we see that \eqref{l4.7:e3} is valid whenever
\[
(\<q_1, q_2> \in u_d) \Rightarrow (\<f(q_1), f(q_2)> \in u_{d_L}),
\]
which follows directly from the commutativity of \eqref{l4.7:e1} and the definition of \(u_d\) and \(u_{d_L}\).
\end{proof}

\begin{proposition}\label{c4.4}
Let \((Q, {\preccurlyeq}_{Q})\) be a totally ordered set with a smallest element \(q_0\). Then there is a \({\preccurlyeq}_{Q}\)-ultrametric \(d \colon Q^2 \to Q\) such that 
\[
d(Q^2) = Q \quad \text{and} \quad {\preccurlyeq}_{Q}^{0} = {\preccurlyeq}_{Q}.
\]
\end{proposition}

\begin{proof}
Let us define a mapping \(d \colon Q^{2} \to Q\) by the rule:
\begin{equation}\label{c4.4:e1}
d(p, q) := \begin{cases}
q_0 & \text{if } p=q,\\
p & \text{if } q \prec_{Q} p,\\
q & \text{if } p \prec_{Q} q.
\end{cases}
\end{equation}
It is clear that \(d\) is symmetric and the equality \(d(p, q) = q_0\) holds if and only if \(p=q\).

Now let \(\<q_1, q_2, q_3>\) be a triple of points of \(Q\). Suppose these points are pairwise different. Since \((Q, {\preccurlyeq}_{Q})\) is totally ordered, there is a permutation 
\[
\begin{pmatrix}
q_1 & q_2 & q_3\\
q_{i_1} & q_{i_2} & q_{i_3}
\end{pmatrix}
\]
such that
\begin{equation}\label{c4.4:e3}
q_{i_1} \prec_{Q} q_{i_3} \prec_{Q} q_{i_2}.
\end{equation}
From \eqref{c4.4:e1} and \eqref{c4.4:e3} it follows that
\[
d(q_{i_1}, q_{i_3}) = q_{i_3} \prec_{Q} q_{i_2} = d(q_{i_1}, q_{i_2}) = d(q_{i_2}, q_{i_3}).
\]
Thus,
\begin{equation}\label{c4.4:e9}
d(q_{i_1}, q_{i_3}) \preccurlyeq d(q_{i_1}, q_{i_2}) = d(q_{i_2}, q_{i_3})
\end{equation}
holds. Analogously, if the number of different points in \(\<q_1, q_2, q_3>\) is two, we can find a permutation such that \(q_{i_1} = q_{i_3} \neq 	q_{i_2}\). Hence, 
\[
d(q_{i_1}, q_{i_3}) = q_{0} \prec_{Q} d(q_{i_1}, q_{i_2}) = d(q_{i_2}, q_{i_3}),
\]
that implies \eqref{c4.4:e9}. For the case when \(q_1 = q_2 = q_3\) holds, \eqref{c4.4:e9} is trivially valid for every permutation
\[
\begin{pmatrix}
q_1 & q_2 & q_3\\
q_{i_1} & q_{i_2} & q_{i_3}
\end{pmatrix}.
\]
Hence, \(d\) is a \({\preccurlyeq}_{Q}\)-ultrametric on \(Q\).

It follows from \eqref{c4.4:e1} that \(d(q_0, q) = q\) holds for every \(q \in Q\). Thus, we have
\begin{equation}\label{c4.4:e4}
d(Q^{2}) = Q.
\end{equation}
To complete the proof it suffices to show that
\begin{equation}\label{c4.4:e5}
{\preccurlyeq}_{Q}^{0} = {\preccurlyeq}_{Q}.
\end{equation}

By definition of \({\preccurlyeq}_{Q}^{0}\), equality \eqref{c4.4:e5} holds if 
\begin{equation}\label{c4.4:e6}
{\preccurlyeq}_{Q}^{0} \supseteq {\preccurlyeq}_{Q}.
\end{equation}
Lemma~\ref{l3.18} and \eqref{c4.4:e4} imply the equality \({\preccurlyeq}_{Q}^{0} = (u_d^t \cup \Delta_{Q})\). Consequently, \eqref{c4.4:e6} is valid if and only if
\begin{equation}\label{c4.4:e7}
(u_d^t \cup \Delta_{Q}) \supseteq {\preccurlyeq}_{Q}.
\end{equation}

Let \(q_1\) and \(q_2\) be some points of \(Q\) and let \(q_1 \preccurlyeq_{Q} q_2\). If there is \(q_3 \in Q\) such that 
\begin{equation}\label{c4.4:e8}
q_1 = d(q_1, q_3) \quad \text{and} \quad q_2 = d(q_1, q_2) = d(q_2, q_3),
\end{equation}
then \(\<q_1, q_2> \in u_d\) holds. If we set \(q_3\) equals to \(q_0\), the smallest element of \((Q, {\preccurlyeq}_{Q})\), then \eqref{c4.4:e8} follows from \(q_1 \preccurlyeq_{Q} q_2\) and \eqref{c4.4:e1}. Thus, the inclusion \(u_d \supseteq {\preccurlyeq}_{Q}\) holds, that implies \eqref{c4.4:e7}.
\end{proof}

\begin{remark}\label{r4.8}
If \(Q\) is finite, \(Q = \{0, 1, \ldots, n\}\), and \({\preccurlyeq}_Q = {\leqslant}\) hold, then the mapping \(d\) defined by \eqref{c4.4:e1} is an ultrametric on \(Q\) for which the ultrametric space \((Q, d)\) is ``as rigid as possible''. Some extremal properties of such spaces and related graph-theoretical characterizations were found in \cite{DPT(Howrigid)}.
\end{remark}

\begin{example}\label{ex4.6}
Let us denote by \(\RR_{0}\) the Cartesian product of \(\RR^{+}\) and the two-points set \(\{0, 1\}\), \(\RR_{0} := \RR^{+} \times \{0, 1\}\), and let \({\preccurlyeq}_{\RR_{0}}\) be the \emph{lexicographical} order on \(\RR_{0}\),
\begin{equation}\label{ex4.6:e1}
\bigl(\<a,b> \preccurlyeq_{\RR_{0}} \<c,d>\bigr) \Leftrightarrow \bigl((a < c) \text{ or } (a = c \text{ and } b = 0 \text{ and } d = 1)\bigr),
\end{equation}
where \(\leqslant\) is the standard order on \(\RR^{+}\). The poset \((\RR_{0}, {\preccurlyeq}_{\RR_{0}})\) is totally ordered. By Proposition~\ref{c4.4}, the mapping \(d \colon \RR_{0}^{2} \to \RR_{0}\), defined by formula~\eqref{c4.4:e1}, is a \({\preccurlyeq}_Q\)-ultrametric and
\begin{equation}\label{ex4.6:e2}
d(\RR_{0}^{2}) = \RR_{0} \quad \text{and} \quad {\preccurlyeq}_{\RR_{0}}^{0} = {\preccurlyeq}_{\RR_{0}}
\end{equation}
hold.

Suppose that there is an ultrametric space \((X, \rho)\) such that \(d\) and \(\rho\) are combinatorially similar. From the definition of combinatorial similarity it follows that there are bijections \(f \colon \rho(X^{2}) \to d(\RR_{0}^{2})\) and \(g \colon \RR_{0} \to X\) such that \(d(x, y) = f(\rho(g(x), g(y)))\) holds for all \(x\), \(y \in \RR_{0}\). Let us consider now the poset \((\rho(X^{2}), {\preccurlyeq}_{\rho})\), where
\begin{equation}\label{ex4.6:e3}
{\preccurlyeq}_{\rho} := u_{\rho}^{t} \cup \Delta_{\rho(X^{2})}.
\end{equation}
By Theorem~\ref{t3.15}, \(\rho\) is a \({\preccurlyeq}_{\rho}\)-ultrametric on \(X\). Moreover, using Lemma~\ref{l3.18} and Theorem~\ref{t4.3} we obtain that \(g \colon \RR_{0} \to X\) is a weak similarity for \(d\) and \(\rho\). Hence, \(f \colon \rho(X^{2}) \to \RR_{0}\) is an isomorphism of \((\RR_{0}, {\preccurlyeq}_{\RR_{0}})\) and \((\rho(X^{2}), {\preccurlyeq}_{\rho})\). Proposition~\ref{p3.17}, Lemma~\ref{l3.18} and \eqref{ex4.6:e3} imply
\begin{equation}\label{ex4.6:e4}
(q_1 \prec_{\RR_{0}} q_2) \Leftrightarrow (f^{-1}(q_1) < f^{-1}(q_2))
\end{equation}
for all \(q_1\), \(q_2 \in \RR_{0}\).

Let us consider now the points
\[
q_i^x := \<x, i> \quad \text{and} \quad q_i^y := \<y, i>, \quad i = 0, 1, \quad x, y \in \RR^{+}.
\]
It follows directly from \eqref{ex4.6:e1} that if \(x < y\), then
\[
q_0^x \prec_{\RR_{0}} q_1^x \prec_{\RR_{0}} q_0^y \prec_{\RR_{0}} q_1^y.
\]
Consequently, 
\begin{equation}\label{ex4.6:e5}
f^{-1}(q_0^x) < f^{-1}(q_1^x) < f^{-1}(q_0^y) < f^{-1}(q_1^y).
\end{equation}
Since \(\QQ^{+} = \RR^{+} \cap \QQ\) is a dense subset of \(\RR^{+}\), for every \(x \in \RR^{+}\) there is \(p^x \in \QQ^{+}\) such that
\begin{equation}\label{ex4.6:e6}
f^{-1}(q_1^x) < p^x < f^{-1}(q_2^x).
\end{equation}
From \eqref{ex4.6:e5} and \eqref{ex4.6:e6} it follows that the mapping
\[
\RR^{+} \ni x \mapsto p^x \in \QQ^{+}
\]
is injective, contrary to the equalities \(|\RR^{+}| = 2^{\aleph_{0}}\) and \(|\QQ^{+}| = \aleph_{0}\). Thus, there are no ultrametrics which are combinatorially similar to \(d\).
\end{example}

\begin{remark}\label{r4.9}
An interesting topological property of the poset \((\RR_{0}, {\preccurlyeq}_{\RR_{0}})\) was found by F.~S.Cater \cite{Cat1999/2000RAE}. We will return to it later in Theorem~\ref{t4.19}.
\end{remark}

Example~\ref{ex4.6} shows that, after replacing \(\aleph_{0}\) by \(2^{\aleph_{0}}\) and \(\QQ^{+}\) by \(\RR^{+}\), Theorem~\ref{t3.7} becomes false. In particular, we have the following proposition.

\begin{proposition}\label{p4.8}
Let \(X\) be a set with \(|X| = 2^{\aleph_{0}}\). Then there is a metric \(d^{*} \colon X^{2} \to \RR^{+}\) such that:
\begin{enumerate}
\item [\((i)\)] If \(\rho\) is an arbitrary ultrametric, then \(\rho\) and \(d^{*}\) are not combinatorially similar;
\item [\((ii)\)] For every \(X_1 \subseteq X\) with \(|X_1| \leqslant \aleph_{0}\), the restriction \(d^{*}|_{X_1^2}\) of \(d^{*}\) is combinatorially similar to an ultrametric.
\end{enumerate}
\end{proposition}

\begin{proof}
Let \(d \colon \RR_{0}^{2} \to \RR_{0}\) be the \({\preccurlyeq}_{\RR_{0}}\)-ultrametric defined in Example~\ref{ex4.6}. The equalities
\begin{equation}\label{p4.8:e1}
|X| = 2^{\aleph_{0}} \quad \text{and} \quad 2^{\aleph_{0}}  = |\RR_{0}|
\end{equation}
imply the existence of a bijection \(g \colon X \to \RR_{0}\). Let \(d_1 \colon X^{2} \to \RR_{0}\) be a \({\preccurlyeq}_{\RR_{0}}\)-ultrametric defined as 
\[
d_1(x, y) = d(g(x), g(y)), \quad x, y \in X.
\]
From~\eqref{p4.8:e1} it follows that \(|d_1(X^{2})| \leqslant 2^{\aleph_{0}}\). Consequently, by statement \((i)\) of Corollary~\ref{c3.17}, there is an usual metric \(d_2\) such that \(d_1\) and \(d_2\) are combinatorially similar. It follows directly from the definition of combinatorial similarity that there is a metric \(d^{*} \colon X^{2} \to \RR^{+}\) which is combinatorially similar to \(d_2\). Thus, \(d^{*}\) and \(d\) are combinatorially similar. 

It is easy to prove that \(d^*\) satisfies conditions \((i)\) and \((ii)\). Indeed, condition \((ii)\) follows from statement \((ii)\) of Corollary~\ref{c3.17}. Furthermore, it was shown in Example~\ref{ex4.6} that there are no ultrametrics which are combinatorially similar to \(d \colon \RR_{0}^{2} \to \RR_{0}\). Consequently, \((i)\) also holds.
\end{proof}

Let \((Q, {\preccurlyeq}_Q)\) be a totally ordered set, and let \(A\), \(B\) be nonempty subsets of \(Q\). We write \(A \prec_{Q} B\) when \(a \prec_{Q} b\) holds for all \(a \in A\) and \(b \in B\). 

The sets \(A\) and \(B\) are \emph{neighboring} if \(A \prec_{Q} B\) or, respectively, \(B \prec_{Q} A\) and there is no \(q \in Q\) such that
\[
A \prec_{Q} \{q\} \quad \text{and} \quad \{q\} \prec_{Q} B
\]
or, respectively,
\[
B \prec_{Q} \{q\} \quad \text{and} \quad \{q\} \prec_{Q} A.
\]

\begin{definition}\label{d4.9}
A totally ordered set \(Q\) is a \(\eta_1\)-set if it has no neighboring subsets which both have a cardinality strictly less than \(\aleph_1\).
\end{definition}

Let \((Q, {\preccurlyeq}_Q)\) and \((L, {\preccurlyeq}_L)\) be posets. An injection \(f \colon Q \to L\) is an \emph{embedding} of \((Q, {\preccurlyeq}_Q)\) in \((L, {\preccurlyeq}_L)\) if
\[
\bigl(q_1 \preccurlyeq_Q q_2\bigr) \Leftrightarrow \bigl(f(q_1) \preccurlyeq_L f(q_2)\bigr)
\]
is valid for all \(q_1\), \(q_2 \in Q\).

A totally ordered set \(L\) is \(\aleph_1\)-\emph{universal} if every totally ordered set \(Q\) with \(|Q| \leqslant \aleph_1\) can be embedded into \(L\).

\begin{lemma}\label{l4.10}
Every \(\eta_1\)-set is \(\aleph_1\)-universal.
\end{lemma}

For the detailed proof of the lemma see, for example, Theorem~20 in~\cite{Ada2018}.

\begin{remark}\label{r4.11}
The above definition of \(\aleph_1\)-universal sets can be naturally extended to arbitrary infinite cardinal number \(\aleph\). The construction of \(\aleph\)-universal posets was studied by many mathematicians (see, for example, \cite{Joh1956PA, Hed1969JoA} and the references therein).
\end{remark}

In the proof of the following theorem we will use the Continuum Hypothesis.

\begin{theorem}\label{t4.11}
Let \(X\) be a nonempty set, let \(\Phi\) be a mapping with \(\dom \Phi = X^{2}\) and \(|\Phi(X^{2})| \leqslant 2^{\aleph_{0}}\), and let \((Q, {\preccurlyeq}_Q)\) be a \(\eta_1\)-set with a smallest element \(q_0\). Then the following conditions are equivalent.
\begin{enumerate}
\item[\((i)\)] \(\Phi\) is combinatorially similar to a \({\preccurlyeq}_Q\)-pseudoultrametric.
\item[\((ii)\)] The mapping \(\Phi\) is symmetric, and the transitive closure \(u_{\Phi}^{t}\) of the binary relation \(u_{\Phi}\) is antisymmetric, and \(\Phi\) is \(a_0\)-coherent for a point \(a_0 \in \Phi(X^{2})\), and, for every triple \(\<x_1, x_2, x_3>\) of points of \(X\), there is a permutation
\[
\begin{pmatrix}
x_1 & x_2 & x_3\\
x_{i_1} & x_{i_2} & x_{i_3}
\end{pmatrix}
\]
such that \(\Phi(x_{i_1}, x_{i_2}) = \Phi(x_{i_2}, x_{i_3})\).
\end{enumerate}
\end{theorem}

\begin{proof}
The validity of \((i) \Rightarrow (ii)\) follows from Theorem~\ref{t3.15}. 

Suppose that \((ii)\) holds. Using Theorem~\ref{t3.15} we obtain that \(\Phi\) is a \({\preccurlyeq}_{\Phi}\)-pseudoultrametric for the partial order
\[
{\preccurlyeq}_{\Phi} := u_{\Phi}^{t} \cup \Delta_{\Phi(X^{2})}
\]
defined on \(\Phi(X^{2})\).

By Lemma~\ref{l3.3} (Szpilrajn), there is an linear order \({\preccurlyeq}_{1}\) on \(\Phi(X^{2})\) such that \({\preccurlyeq}_{\Phi} \subseteq {\preccurlyeq}_{1}\). Consequently, \(\Phi\) is also a \({\preccurlyeq}_{1}\)-pseudoultrametric. The inequality \(|\Phi(X^{2})| \leqslant 2^{\aleph_{0}}\) holds. The Continuum Hypothesis, \(2^{\aleph_{0}} = \aleph_1\), and the last inequality imply the inequality \(|\Phi(X^{2})| \leqslant \aleph_1\). By Lemma~\ref{l4.10}, the \(\eta_1\)-set \((Q, {\preccurlyeq}_Q)\) is \(\aleph_1\)-universal. It is easy to prove that there is an embedding \(f \colon \Phi(X^{2}) \to Q\) of \((\Phi(X^{2}), {\preccurlyeq}_{1})\) in \((Q, {\preccurlyeq}_{Q})\) such that \(f(a_0) = q_0\). Then the mapping 
\[
X^2 \xrightarrow{\Phi} \Phi(X^{2}) \xrightarrow{f} Q
\]
is a \({\preccurlyeq}_{Q}\)-pseudoultrametric and this mapping is combinatorially similar to \(\Phi\).
\end{proof}

The following definition can be found in \cite[pp.~57--58]{Kel1975S}.

\begin{definition}\label{d4.13}
Let \((Q, {\preccurlyeq}_{Q})\) be a totally ordered set with \(|Q| > 1\). A topology \(\tau\) with a subbase consisting of all sets of the form
\[
\{q \in Q \colon q \prec_Q a\} \quad \text{or} \quad \{q \in Q \colon a \prec_Q q\}
\]
for some \(a \in Q\) is the order topology on \(Q\). In this case we say that \(\tau\) is the \({\preccurlyeq}_{Q}\)-topology for short.
\end{definition}

Recall that a topological space is second countable if it has a countable or finite base.

\begin{lemma}\label{l4.14}
Let \((Q, {\preccurlyeq}_{Q})\) be a totally ordered set with \(|Q| > 1\). Then the following conditions are equivalent.
\begin{enumerate}
\item [\((i)\)] The \({\preccurlyeq}_{Q}\)-topology is second countable.
\item [\((ii)\)] The poset \((Q, {\preccurlyeq}_{Q})\) is isomorphic to a subposet of \((\RR^{+}, \leqslant)\).
\end{enumerate}
\end{lemma}

This lemma is a simple modification of Theorem~II from paper~\cite{Cat1999/2000RAE} of F.~S.~Cater.

\begin{theorem}\label{t4.15}
Let \((Q, {\preccurlyeq}_{Q})\) be a totally ordered set satisfying \(|Q| > 1\) and having the smallest element \(q_0\). Then the following conditions are equivalent.
\begin{enumerate}
\item [\((i)\)] The \({\preccurlyeq}_{Q}\)-topology is second countable.
\item [\((ii)\)] For every \({\preccurlyeq}_{Q}\)-pseudoultrametric \(d\) there is a pseudoultrametric \(\rho\) such that \(d\) and \(\rho\) are weakly similar.
\item [\((iii)\)] For every \({\preccurlyeq}_{Q}\)-pseudoultrametric \(d\) there is a pseudoultrametric \(\rho\) such that \(d\) and \(\rho\) are combinatorially similar.
\end{enumerate}
\end{theorem}

\begin{proof}
It is easy to see that \((i)\), \((ii)\) and \((iii)\) are equivalent if \(|Q| = 2\). Suppose \(|Q| \geqslant 3\) holds.

\((i) \Rightarrow (ii)\). Let the \({\preccurlyeq}_{Q}\)-topology be second countable, let \(X\) be a nonempty set and let \(d \colon X^{2} \to Q\) be a \({\preccurlyeq}_{Q}\)-pseudoultrametric. Write \(Q_0 := Q \setminus \{q_0\}\) and \({\preccurlyeq}_{Q_0} := Q_0^2 \cap {\preccurlyeq}_{Q}\). The inequality \(|Q| \geqslant 3\) implies \(|Q_0| > 1\). The \({\preccurlyeq}_{Q_0}\)-topology coincides with the topology induced on \(Q_0\) by \({\preccurlyeq}_{Q}\)-topology. Consequently, the \({\preccurlyeq}_{Q_0}\)-topology is also second countable. Hence, by Lemma~\ref{l4.14}, there is an isomorphism \(f \colon Q_0 \to A_0\) of the posets \((Q_0, {\preccurlyeq}_{Q_0})\) and \((A_0, \leqslant)\), where \(A_0 \subseteq (0, \infty)\) and \(\leqslant\) is the standard order on \(\RR\). Write \(A := A_0 \cup \{0\}\). The function \(f^{*} \colon Q \to A\),
\[
f^{*}(q) = \begin{cases}
0 & \text{if } q = q_0,\\
f(q) & \text{if } q\neq q_0,
\end{cases}
\]
is an isomorphism of \((Q, {\preccurlyeq}_{Q})\) and \((A, \leqslant)\). Let \(\rho \colon X^{2} \to \RR^{+}\) be defined as
\[
\rho(x, y) = f^{*}(d(x, y)), \quad x, y \in X.
\]
Then \(\rho\) is a pseudoultrametric on \(X\) and the identical mapping \hbox{\(X \xrightarrow{\operatorname{id}} X\)} is a weak similarity for \(d\) and \(\rho\).

\((ii) \Rightarrow (iii)\). The validity of this implication follows from Proposition~\ref{p3.16}.

\((iii) \Rightarrow (i)\). Suppose condition \((iii)\) holds. By Proposition~\ref{c4.4}, there is a \({\preccurlyeq}_{Q}\)-ultrametric \(d \colon Q^2 \to Q\) satisfying the equalities \(d(Q^2) = Q\) and \({\preccurlyeq}_{Q}^{0} = {\preccurlyeq}_{Q}\).

Let \(\rho \colon X^{2} \to \RR^{+}\) be a pseudoultrametric such that \(\rho\) and \(d\) are combinatorially similar. Write \(L := \rho(X^2)\) and \({\preccurlyeq}_{L} := {\leqslant} \cap L^2\). Then the \(L\)-pseudoultrametric \(\rho_L \colon X^{2} \to L\),
\[
\rho_L(x, y) = \rho(x, y), \quad x, y \in X,
\]
is also combinatorially similar to \(d\). By Lemma~\ref{l4.7}, \(d\) and \(\rho_L\) are weakly similar. Using Definition~\ref{d3.13} we obtain that \((Q, {\preccurlyeq}_{Q})\) is isomorphic to the subposet \((L, {\preccurlyeq}_{L})\) of \((\RR^{+}, {\leqslant})\). Hence, by Lemma~\ref{l4.14} (Cater), the \({\preccurlyeq}_{Q}\)-topology is second countable.
\end{proof}

Recall that a topological space \((X, \tau)\) is said to be separable if there is a set \(A \subseteq X\) such that \(|A| \leqslant \aleph_{0}\) and \(A \cap U \neq \varnothing\) for every nonempty set \(U \in \tau\).

In what follows we denote by \((\RR_{0}, {\preccurlyeq}_{\RR_{0}})\) the totally ordered set constructed in Example~\ref{ex4.6}.

The next lemma is a part of Theorem~III \cite{Cat1999/2000RAE}.

\begin{lemma}[Cater]\label{l4.18}
Let \((Q, {\preccurlyeq}_{Q})\) be a totally ordered set with \(|Q| > 1\). Then the following conditions are equivalent.
\begin{enumerate}
\item [\((i)\)] The \({\preccurlyeq}_{Q}\)-topology is separable.
\item [\((ii)\)] The poset \((Q, {\preccurlyeq}_{Q})\) is isomorphic to a subposet of \((\RR_{0}, {\preccurlyeq}_{\RR_{0}})\).
\end{enumerate}
\end{lemma}

\begin{theorem}\label{t4.19}
Let \((Q, {\preccurlyeq}_{Q})\) be a totally ordered set having a smallest element and satisfying the inequality \(|Q| > 1\). Then the following conditions are equivalent.
\begin{enumerate}
\item [\((i)\)] The \({\preccurlyeq}_{Q}\)-topology is separable.
\item [\((ii)\)] For every \({\preccurlyeq}_{Q}\)-pseudoultrametric \(d\) there is a \({\preccurlyeq}_{\RR_{0}}\)-pseudo\-ultra\-metric \(\rho\) such that \(d\) and \(\rho\) are weakly similar.
\item [\((iii)\)] For every \({\preccurlyeq}_{Q}\)-pseudoultrametric \(d\) there is a \({\preccurlyeq}_{\RR_{0}}\)-pseudo\-ultra\-metric \(\rho\) such that \(d\) and \(\rho\) are combinatorially similar.
\end{enumerate}
\end{theorem}

Using Lemma \ref{l4.18} instead of Lemma \ref{l4.14} we can prove this theorem similarly to Theorem~\ref{t4.15}.

The following theorem gives us some necessary and sufficient conditions under which a mapping is combinatorially similar to a pseudoultrametric, and it can be considered as a main result of the section.

\begin{theorem}\label{t4.20}
Let \(X\) be a nonempty set and let \(\Phi\) be a mapping with \(\dom \Phi = X^{2}\). Then the following conditions are equivalent.
\begin{enumerate}
\item [\((i)\)] \(\Phi\) is combinatorially similar to pseudoultrametric.
\item [\((ii)\)] There is \(b_0 \in \Phi(X^{2})\) such that \(\Phi(x,x) = b_0\) holds for every \(x \in X\), and the binary relation 
\begin{equation}\label{t4.20:e1}
{\preccurlyeq}_{\Phi} := u_{\Phi}^{t} \cup \Delta_{\Phi(X^{2})}
\end{equation}
is a partial order on \(\Phi(X^{2})\), and \(b_0\) is the smallest element of \((\Phi(X^{2}), {\preccurlyeq}_{\Phi})\), and \(\Phi\) is a \( {\preccurlyeq}_{\Phi}\)-pseudo\-ultra\-metric on \(X\), and there is a linear order \({\preccurlyeq}\) on \(\Phi(X^{2})\) such that
\begin{equation}\label{t4.20:e2}
{\preccurlyeq}_{\Phi} \subseteq {\preccurlyeq}
\end{equation}
holds, and \((\Phi(X^{2}), {\preccurlyeq})\) is isomorphic to a subposet of \((\RR^{+}, \leqslant)\).
\item [\((iii)\)] The mapping \(\Phi\) is symmetric, and there is \(a_0 \in \Phi(X^{2})\) for which \(\Phi\) is \(a_0\)-coherent, and, for every triple \(\<x_1, x_2, x_3>\) of points of \(X\), there is a permutation
\[
\begin{pmatrix}
x_1 & x_2 & x_3\\
x_{i_1} & x_{i_2} & x_{i_3}
\end{pmatrix}
\]
such that \(\Phi(x_{i_1}, x_{i_2}) = \Phi(x_{i_2}, x_{i_3})\), and there is a linear order \({\preccurlyeq}\) on \(\Phi(X^{2})\) such that \(a_0\) is the smallest element of \((\Phi(X^{2}), {\preccurlyeq})\) and \(u_{\Phi} \subseteq {\preccurlyeq}\) holds, and \((\Phi(X^{2}), {\preccurlyeq})\) is isomorphic to a subposet of \((\RR^{+}, \leqslant)\).
\end{enumerate}
\end{theorem}

\begin{proof}
\((i) \Rightarrow (ii)\). Let \((i)\) hold. Then using Theorem \ref{t3.15} we see that condition \((ii)\) is valid whenever there is a linear order \(\preccurlyeq\) on \(\Phi(X^{2})\) such that \eqref{t4.20:e2} holds and \((\Phi(X^{2}), {\preccurlyeq})\) is  isomorphic to a subposet of \((\RR^{+}, \leqslant)\). 

By condition~\((i)\), there are a set \(Y\) and a pseudoultrametric \(\rho \colon Y^{2} \to \RR^{+}\) such that \(\Phi\) and \(\rho\) are combinatorially similar. Write
\begin{equation}\label{t4.20:e3}
{\preccurlyeq}_{\rho} := u_{\rho}^{t} \cup \Delta_{\rho(Y^{2})}.
\end{equation}
From Lemma~\ref{l3.18} it follows that \(\rho\) is a \({\preccurlyeq}_{\rho}\)-pseudo\-ultra\-metric. Since \(\Phi\) and \(\rho\) are combinatorially similar, there exists a bijection \(g \colon X \to Y\) such that \(g\) is combinatorial similarity for \(\Phi\) and \(\rho\). Now using Theorem~\ref{t4.3}, and \eqref{t4.20:e1}, and \eqref{t4.20:e3} we see that \(g\) is a weak similarity for \(\Phi\) and \(\rho\). Consequently, there is an order isomorphism 
\[
f \colon \Phi(X^{2}) \to \rho(Y^{2})
\]
of posets \((\Phi(X^{2}), {\preccurlyeq}_{\Phi})\) and \((\rho(Y^{2}), {\preccurlyeq}_{\rho})\). By Proposition~\ref{p3.17} and Lemma~\ref{l3.18}, we obtain that
\[
(\gamma_1 \preccurlyeq_{\rho} \gamma_2) \Rightarrow (\gamma_1 \leqslant \gamma_2)
\]
is valid for all \(\gamma_1\), \(\gamma_2 \in \rho(Y^{2})\).

Let us define a binary relation \(\preccurlyeq\) by the rule:
\[
(\<g_1, g_2> \in {\preccurlyeq}) \Leftrightarrow (\<g_1, g_2> \in \Phi(X^{2}) \times \Phi(X^{2}) \text{ and } (f(g_1) \leqslant f(g_2)))
\]
Then \({\preccurlyeq}\) is a linear order satisfying all desirable conditions.

\((ii) \Rightarrow (i)\). Suppose \((ii)\) holds. Then \(\Phi\) is a \({\preccurlyeq}_{\Phi}\)-pseudo\-ultra\-metric on \(X\) and there is an injection \(f \colon \Phi(X^{2}) \to \RR^{+}\) such that
\[
(b_1 \preccurlyeq_{\Phi} b_2) \Rightarrow (f(b_1) \leqslant f(b_2))
\]
holds for all \(b_1\), \(b_2 \in \Phi(X^{2})\). Since \(b_0\) is the smallest element of the poset \((\Phi(X^{2}), {\preccurlyeq}_{\Phi})\), the function \(f^{*} \colon \Phi(X^{2}) \to \RR^{+}\) defined as
\[
f^{*}(b) = f(b) - f(b_0)
\]
is nonnegative and isotone, and satisfies the condition
\[
(f^{*}(b) = 0) \Leftrightarrow (b = b_0)
\]
for every \(b \in \Phi(X^{2})\). Proposition~\ref{p3.23} implies that \(f^{*} \circ \Phi\) is a pseudoultrametric on \(X\). From Definition~\ref{d2.17} it directly follows that \(\Phi\) and \(f^{*} \circ \Phi\) are combinatorially similar.

The validity of the equivalence \((ii) \Leftrightarrow (iii)\) follows from Theorem~\ref{t3.15}. We only note that \(u_{\Phi}^{t}\) is antisymmetric if and only if there is a partial order \({\preccurlyeq}'\) such that \({\preccurlyeq}' \supseteq u_{\Phi}\).
\end{proof}

The proof of the following corollary is similar to prove of Theorem~\ref{t4.20}.

\begin{corollary}\label{c4.22}
Let \(X\) be a nonempty set and let \(\Phi\) be a mapping with \(\dom \Phi = X^{2}\). Then the following conditions are equivalent.
\begin{enumerate}
\item [\((i)\)] \(\Phi\) is combinatorially similar to ultrametric.
\item [\((ii)\)] There is \(b_0 \in \Phi(X^{2})\) such that \(\Phi^{-1}(b_0) = \Delta_{X}\), and the binary relation 
\[
{\preccurlyeq}_{\Phi} := u_{\Phi}^{t} \cup \Delta_{\Phi(X^{2})}
\]
is a partial order on \(\Phi(X^{2})\), and \(b_0\) is the smallest element of \((\Phi(X^{2}), {\preccurlyeq}_{\Phi})\), and \(\Phi\) is a \( {\preccurlyeq}_{\Phi}\)-ultra\-metric on \(X\), and there is a linear order \({\preccurlyeq}\) on \(\Phi(X^{2})\) such that
\[
{\preccurlyeq}_{\Phi} \subseteq {\preccurlyeq}
\]
holds, and \((\Phi(X^{2}), {\preccurlyeq})\) is isomorphic to a subposet of \((\RR^{+}, \leqslant)\).
\item [\((iii)\)] The mapping \(\Phi\) is symmetric, and there is \(a_0 \in \Phi(X^{2})\) for which \(\Phi^{-1}(a_0) = \Delta_{X}\) holds, and, for every triple \(\<x_1, x_2, x_3>\) of points of \(X\), there is a permutation
\[
\begin{pmatrix}
x_1 & x_2 & x_3\\
x_{i_1} & x_{i_2} & x_{i_3}
\end{pmatrix}
\]
such that \(\Phi(x_{i_1}, x_{i_2}) = \Phi(x_{i_2}, x_{i_3})\), and there is a linear order \({\preccurlyeq}\) on \(\Phi(X^{2})\) such that \(a_0\) is the smallest element of \((\Phi(X^{2}), {\preccurlyeq})\) and \(u_{\Phi} \subseteq {\preccurlyeq}\) holds, and \((\Phi(X^{2}), {\preccurlyeq})\) is isomorphic to a subposet of \((\RR^{+}, \leqslant)\).
\end{enumerate}
\end{corollary}

In connection with Theorem \ref{t4.20} and Corollary \ref{c4.22}, the following problem naturally arises.

\begin{problem}\label{pr4.21}
Describe (up to order-isomorphism) the partially ordered sets \((Q, {\preccurlyeq}_{Q})\) which admit extensions to totally ordered sets \((Q, {\preccurlyeq})\) such that \((Q, {\preccurlyeq})\) is order-isomorphic to a subposet of \((\RR^{+}, \leqslant)\).
\end{problem}

We do not discuss this problem in details but formulate the following conjecture.

\begin{conjecture}\label{con4.24}
The following conditions are equivalent.
\begin{enumerate}
\item [\((i)\)] A poset \((Q, {\preccurlyeq}_{Q})\) admits an extension to totally ordered set \((Q, {\preccurlyeq})\) such that \((Q, {\preccurlyeq})\) is order-isomorphic to a subposet of \((\RR^{+}, \leqslant)\).
\item [\((ii)\)] The inequality \(|Q| \leqslant 2^{\aleph_{0}}\) holds and every totally ordered subposet of \((Q, {\preccurlyeq}_{Q})\) can be embedded into \((\RR^{+}, \leqslant)\).
\end{enumerate}
\end{conjecture}


\end{document}